\documentclass[10pt]{article}
\usepackage{amsmath}
\usepackage{mathrsfs}
\usepackage{amssymb}
\usepackage{latexsym}
\usepackage{amsthm}
\usepackage{amsmath,amssymb}
\usepackage{bm}
\usepackage{ascmac}
\usepackage{enumerate}
\usepackage[margin=20truemm]{geometry}
\theoremstyle{plain}
\usepackage[dvipdfmx]{graphicx}

\newtheorem{prop}{Proposition}[section]
\newtheorem{lem}{Lemma}[section]
\newtheorem{thm}{Theorem}[section]
\newtheorem{cor}{Corollary}[section]

\newtheorem{remark}{Remark}[section]

\numberwithin{equation}{section}

\newcommand{\vare}{\varepsilon}

\allowdisplaybreaks

\title{Limiting distributions of generalized money exchange models}
\author{
Hironobu Sakagawa
\thanks{
Department of Mathematics, 
Faculty of Science and Technology, Keio University, 
3-14-1 Hiyoshi, Kouhoku-ku, Yokohama 223-8522, JAPAN. 
{\it E-mail address\/}: sakagawa{\char'100}math.keio.ac.jp} 
}

\begin{document}

\maketitle


\begin{abstract}
The ``Money Exchange Model" is a type of agent-based simulation model 
used to study how wealth distribution and inequality evolve 
through monetary exchanges between individuals. 
The primary focus of this model is to identify the limiting 
wealth distributions that emerge at the macroscopic level, 
given the microscopic rules governing the exchanges among agents. 
In this paper, we formulate generalized versions of 
the immediate exchange model and the uniform saving model 
both of which are types of money exchange models, 
as discrete-time interacting particle systems 
and characterize their stationary distributions. 
Furthermore, we prove that under appropriate scaling, 
the asymptotic wealth distribution converges to a Gamma distribution 
or an exponential distribution for both models. 
The limiting distribution depends on the weight 
function that affects the probability distribution of 
the number of coins exchanged by each agent. 
In particular, our results provide a mathematically rigorous formulation 
and generalization of the assertions previously predicted in studies 
based on numerical simulations and heuristic arguments.

\end{abstract}

\begin{center}{\small\bfseries Key words.}
{\small Econophysics, interacting particle system, 
stationary distribution, equivalence of ensembles, 
local limit theorem.} 
\end{center}



\section{Introduction}
\subsection{Money exchange models}
The ``Money Exchange Model" is a type of agent-based simulation model used to study how wealth distribution and inequality evolve through monetary exchanges between individuals. This model has been extensively studied in the field of econophysics, particularly by applying ideas from statistical physics, 
where the money exchange 
is viewed as analogous to the transfer of energy or particles.

Consider an economy consisting of a finite number of agents. The typical process proceeds as follows:
\begin{enumerate}[(1)]
\item 
Initial conditions: Assign each agent a random or equal amount of money.
\item
Selecting pairs of agents: Randomly select pairs of agents who will 
exchange money.
\item Money exchange: 
The selected pair exchanges money according to specific rules.
\item
Iteration: Repeat (2) and (3) multiple times and observe how the distribution of money in the system changes. 
\end{enumerate}
\noindent
The primary focus of this model is to identify the limiting wealth distributions that emerge at the macroscopic level, given the microscopic rules governing the exchanges among agents. Depending on the rules of exchange, the limiting distribution is often expected to take the form of a Gamma or an exponential distribution. In the field of econophysics, various studies have been conducted through numerical simulations and other methods 
(cf. \cite{CCC}, \cite{YR09} and references theirin). 
However, mathematically rigorous studies of these models remain relatively few.

We begin by describing several specific models. 
The first one is the {\em immediate exchange model} 
proposed in \cite{HP14}. 
In this model, the wealth of each agent is 
represented by a real-valued variable. 
At each time step, 
two agents are randomly chosen and give a random fraction of 
their wealth to each other. 
The fraction is determined by independent uniformly 
distributed random variables on the interval $[0, 1]$. 
\cite{HP14} studied the model through numerical simulations, 
and later, \cite{K15} analytically explored the infinite population version. 
As a more realistic microscopic model that includes spatial structure 
in the form of local interactions, 
\cite{LR18} formulated the corresponding discrete version 
as an interacting particle system.
We briefly explain their model and result. 
Consider a finite connected graph 
$\mathcal{G}=(\mathcal{V}, \mathcal{E})$. 
Each site $x\in \mathcal{V}$ corresponds to an agent, 
and the economy is represented by the set of agents $\mathcal{V}$. 
The population size is given by $|\mathcal{V}|=N$.
The amount of money each agent holds is represented by 
the number of coins, where 
$M_n(x)$ represents the number of coins that agent 
$x \in \mathcal{V}$ possesses at time $n$. 
The edge set $\mathcal{E}$ represents a social network in which 
only agents connected by an edge can interact to exchange coins. 
At each time step, an edge $e=\{x, y\}\in \mathcal{E}$ 
is chosen uniformly at random from $\mathcal{E}$. 
Agents $x$ and $y$ independently and uniformly select 
a random number of their coins to give to each other. 
Thus, $\{(M_n(x))_{x\in \mathcal{V}}\}_{n\geq 0}$ constitutes 
a time-homogeneous Markov chain.  
In particular, the total number of coins is conserved 
in this process, and we denote this total by $L$. 
\cite{LR18} proved that the stationary distribution 
$\mu_{N, L}$ 
for this Markov chain uniquely exists and 
gave the following explicit representation. 
\begin{align}\label{lr1}
\begin{split}
\lim_{n\to \infty}P(M_n(x) = c) 
& = \mu_{N, L}\bigl( \{\xi\in \mathbb{Z}_+^{\mathcal{V}}; 
\xi(x)= c \}\bigr) \\
& = 
\frac{(c+1) 
\Bigl(
\begin{array}{c}
L-c + 2N -3  \\
2N-3
\end{array}
\Bigr)}
{\Bigl(
\begin{array}{c}
L+ 2N-1  \\
2N-1
\end{array}
\Bigr)}, 
\end{split}
\end{align}
for every 
$x\in \mathcal{V}$ 
and $c\in \{0,1, 2, \cdots, L\}$. 
The first equality is a consequence of the Markov chain 
convergence theorem. 
Then, by applying formal calculations to the 
right-hand side, the authors obtained the following approximation. 
\begin{equation}\label{app}
\lim_{n\to \infty}P(M_n(x) = c) 
\approx \frac{4c}{T^2}e^{-\frac{2c}{T}}, 
\end{equation}
for large enough $N$ and $T=\frac{L}{N}$. 
In this context 
$T$ represents the average number of coins per agent and is 
called {\em money temperature} 
by analogy with the temperature in physics and 
the limit $N\to\infty$ and $T=\frac{L}{N}\to \infty$ 
are called  {\em large population and 
large money temperature limit}. 
The right-hand side of (\ref{app}) is a probability density 
function of the Gamma distribution with mean $T$ and shape parameter two, 
and this is consistent with the predicted result by 
\cite{HP14} and \cite{K15}. 
As related results, 
\cite{CRR14} and \cite{GRS16} examined the duality 
between the real-valued model and the discrete state version 
in a continuous time setting. 
In addition, \cite{GRS16} considered the case where 
the exchange fraction is determined by a Beta distribution.

We also present two models with different exchange rules 
that we address in this paper. 

\smallskip

\noindent 
{{{\em Uniform saving model:}}}\ 
Agents $x$ and $y$ independently save a random number of 
their coins according to a uniform distribution. 
The remaining coins are then pooled and uniformly redistributed 
between the two agents. 
As for the immediate exchange model, 
the limiting distribution for this model is predicted 
to be a Gamma distribution 
with shape parameter two (cf. \cite{CC00}, \cite{PCK04}), and 
\cite{LR18} obtained the same approximation as (\ref{app}). 

\smallskip

\noindent 
{{{\em Uniform reshuffling model:}}}\ 
All the coins agents $x$ and $y$ possess are pooled 
and uniformly redistributed between the two agents. 
Different from the above two models, 
the limiting distribution for this model 
is predicted to be an exponential distribution with mean $T$ 
(cf. \cite{DY00}). 
\cite{LR18} obtained the corresponding approximation 
of the same form as (\ref{app}). 

\smallskip


So far, \cite{LR18} has formulated several money exchange models 
as Markov chains with spatial structure, 
characterized their stationary distributions and further obtained 
their formal approximations. 
The conclusions explain the distribution of wealth in a sense, 
as predicted by numerical simulations and other methods. 
However, the approximation $`` \approx "$ in (\ref{app}) 
is not mathematically valid. 
The left-hand side is defined only for 
non-negative integers $c$, 
while the right-hand represents a probability density function 
on $\mathbb{R}_+$. 
The derivation of the approximation (\ref{app}), 
although merely a formal calculation, 
seems to depend on assumptions that are not fully clarified.
Specifically, in the proofs of Theorems 1, 2 and 3 in \cite{LR18}, 
$L+1, L+2, \cdots, L+N$ are replaced with $L$ 
for sufficiently large $L$ and $N$ satisfying $N \ll L$. 
On the other hand, $L-c+1, L-c+2, \cdots, L-c+N$ are 
replaced with $L-c$, rather than $L$, even when $c\ll L$. 
Additionally, $c+1$ has been conveniently replaced with $c$.
As a matter of fact, 
it seems unnatural to consider the approximation of 
$\mu_{N, L}( \{\xi\in \mathbb{Z}_+^{\mathcal{V}}; 
\xi(x)= c \} )$ for each $c$ 
since the average number of coins per agent diverges 
in the limit $T=\frac{L}{N} \to \infty$. 
The money exchange model describes the microscopic movement of money, 
however, 
our goal is to derive the macroscopic distribution of wealth 
in the limit $N\to \infty$ and $\frac{L}{N}\to \infty$. 
In order to do that, 
we should analyze the convergence of $\mu_{N, L}$ 
under appropriate scaling. 
Therefore, the primary objective of this paper is to provide 
a mathematically rigorous justification of the approximation 
(\ref{app}) and to clearly demonstrate the convergence 
of the wealth distribution.
Furthermore, 
we generalize the rules of money exchange as follows: 
\begin{itemize}
\item
Randomly select the number of coins to pass or save 
based on a probability distribution that depends on the number of coins, 
rather than using the uniform distribution. 
\item
Allow the exchange or redistribution of coins among randomly 
selected groups of three or more agents.
\end{itemize}
These generalizations appear natural 
from both mathematical and economic perspectives. 
We note that 
as a generalization of the exchange rules in the immediate exchange model, 
\cite{RS17} considered a broader class of models 
where mass is split, exchanged and merged. 
In \cite{L17}, \cite{LR19} and \cite{LR24}, 
the authors formulated other money exchange models 
as Markov chains and 
studied their stationary distributions and 
formal approximations in a manner similar to 
(\ref{lr1}) and (\ref{app}). 
Also, the mixing time has been studied recently for 
the binomial splitting model and the symmetric beta-binomial splitting model, 
which are variations of the uniform reshuffling model 
(cf. \cite{PR23}, \cite{QS23}). 



Before introducing our models and results we prepare several notations. 
In the following, $\mathbb{Z}_+=\{0, 1, 2, \cdots \}$ denotes the set of 
non-negative integers and 
$\mathbb{N}=\{1, 2, 3, \cdots \}$ denotes the set of positive integers. 
$[a]$ denotes the integral part of $a>0$ and 
we set 
$a \vee b:= \max\{a, b\}$ and 
$a \wedge b:= \min\{a, b\}$ for $a, b \in \mathbb{R}$. 
For a finite set $A$, $|A|$ denotes its cardinality. 
For two sequences of positive numbers $\{a_n\}$ and $\{b_n\}$, 
$a_n\sim b_n$ means that $\lim\limits_{n\to \infty} 
\frac{a_n}{b_n}=1$. 
A function $f$ on $\mathbb{R}^{\mathbb{Z}}$ is 
called local if it depends only on finitely many coordinates. 
For a probability measure $\mu$, $E^\mu[\, \cdot\, ]$ denotes 
the expectation with respect to $\mu$. 

\subsection{Model description and results}

Let us state our model precisely. 
We adopt the standard notations commonly used 
in the study of interacting particle systems.
For $A \subset \mathbb{Z}$ and $L \in \mathbb{N}$, we define 
the configuration space 
\begin{equation*}
\Omega(A, L)=\bigl\{
\eta = \{\eta(x)\}_{x\in A} \in \mathbb{Z}_+^A; 
\sum\limits_{x\in A}\eta(x) =L
\bigr\}.
\end{equation*}
When $A=\Lambda_N:= \{1, 2, \cdots, N\}$, 
$\Omega(A, L)$ is denoted by $\Omega_N(L)$. 
Consider now an economy populated by many agents. 
Each site $x\in \Lambda_N$ corresponds to an agent and 
we assume that the economy can be represented 
by the set of agents $\Lambda_N$. 
$N$ corresponds to the population size. 
For each $\eta \in \Omega_N(L)$, we interpret
$\eta(x)$, $x\in \Lambda_N$ 
not as the number of particles, 
but as the number of coins held by agent $x$. 
$\rho$ denotes a probability distribution on 
$\mathcal{D}_N= \{ A\subset \Lambda_N; |A|\geq 2\}$, 
namely $\rho(A)\geq 0$ for every $A \subset \Lambda_N$ 
with $|A|\geq 2$ and $\sum\limits_{A\in \mathcal{D}_N}
\rho(A)=1$. 
This represents the distribution 
that determines which agent handles the money exchange 
at each time step. 
We also take a non-negative function 
$g ( \not\equiv 0)$ defined on $\mathbb{Z}_+$. 
Now, we introduce three money exchange models.

\smallskip

\noindent 
{{{\em Immediate exchange model:}}}\ 
Let $\{X_n\}_{n \geq 0}$ be a time-homogeneous Markov chain 
on the state space $\Omega_N(L)$. 
For given $X_n=\xi \in \Omega_N(L)$, the configuration 
$X_{n+1}$ is determined from the following rule. 
\begin{enumerate}[(1)]
\item
Choose a set $A\in \mathcal{D}_N$ according to the distribution 
$\rho$. 
\item
For given $\xi$, 
let $\{c(x)\}_{x\in A}$ be independent 
random variables whose distributions are given by 
\begin{equation}\label{gg}
P(c(x)=k)= \frac{1}{G(\xi(x))}g(k),\ 
k\in \{ 0, 1, 2, \cdots, \xi(x)\}, 
\end{equation}
for $x\in A$ where we set 
$G(k) = \sum\limits_{j=0}^k g(j)$, $k\in \mathbb{Z}_+$. 
\item
Choose a permutation $\sigma \in \mathcal{S}_A$ uniformly random, 
namely with probability 
$\frac{1}{|\mathcal{S}_A|}=\frac{1}{|A|!}$ 
where $\mathcal{S}_A$ denotes the set of all permutations of $A$. 
\item
For given $\xi$ and realizations $A$, $\{c(x)\}_{x\in A}$ and $\sigma$, 
define $X_{n+1}$ by 
\begin{equation*}
X_{n+1}(z) = 
\begin{cases}
\xi(z) -c(z) + c(\sigma^{-1}(z)) & \text{ if } z \in A, \\
\xi(z) & \text{ if } z \notin A. \\
\end{cases}
\end{equation*}
\end{enumerate}
This dynamics can be interpreted as follows: 
At each time step, a money exchange occurs between agents 
in a randomly chosen set $A$. 
$c(x)$ represents the number of coins that 
agent $x$ transfers, which is determined by a 
probability distribution dependent on 
the weight function $g$ and $\xi(x)$, 
the number of coins that agent $x$ possesses. 
According to a randomly chosen permutation $\sigma\in \mathcal{S}_A$, 
each agent $x \in A$ passes $c(x)$ coins to agent $\sigma(x)\in A$. 
In this process, the total number of coins remains conserved. 
It is worth noting that the model studied in \cite{LR18} 
corresponds to the case where $g$ is a constant function 
and $\rho$ is the uniform distribution on an edge set of $\Lambda_N$. 
In this case, the distribution (\ref{gg}) matches 
the uniform distribution on $\{0, 1, 2, \cdots, \xi(x)\}$ and 
the money exchange occurs between two agents connected 
by an edge in $\Lambda_N$. 
To be more precise, since the permutation $\sigma\in \mathcal{S}_A$ 
can include the identity permutation, 
our model can be viewed as the lazy version of their model in this case.


\smallskip

\noindent
{{{\em Random saving model:}}}\ 
Let $\{Y_n\}_{n \geq 0}$ be a time-homogeneous Markov chain 
on the state space $\Omega_N(L)$. 
For given $Y_n=\xi$, the configuration 
$Y_{n+1}$ is determined from the following rule. 
\begin{enumerate}[(1)]
\item
Choose a set $A\in \mathcal{D}_N$ according to the distribution 
$\rho$. 
\item
For given $\xi$, let $\{c(x)\}_{x\in A}$ be independent 
random variables whose distributions are given by 
\begin{equation*}
P(c(x)=k)= \frac{1}{G(\xi(x))}g(k),\ 
k\in \{ 0, 1, 2, \cdots, \xi(x)\}, 
\end{equation*}
for $x\in A$. 
\item
For given $\xi$ and $c=\{c(x)\}_{x\in A}$, 
choose a configuration $d =\{d(x)\}_{x\in A} 
\in \Omega(A, S_A(\xi)-S_A(c))$ uniformly random, 
namely with probability 
$\frac{1}{|\Omega(A, S_A(\xi)-S_A(c))|}$ 
where we set 
$ S_A(\xi)= \sum\limits_{x\in A}\xi(x)$ 
for $\{\xi(x)\}_{x\in A}$. 
\item
For given $\xi$ and realizations $A$, $\{c(x)\}_{x\in A}$ and 
$\{d(x)\}_{x\in A}$, 
define $Y_{n+1}$ by 
\begin{equation*}
Y_{n+1}(z) = 
\begin{cases}
c(z) + d(z) & \text{ if } z \in A, \\
\xi(z) & \text{ if } z \notin A. \\
\end{cases}
\end{equation*}
\end{enumerate}
Note that 
in contrast to the immediate exchange model, 
the random variable $c(x)$ represents 
how many coins the agent $x$ to save. 
When a money exchange occurs within set $A$, 
each agent $x\in A$ offers $\xi(x)-c(x)$ coins. 
These coins are then pooled and redistributed among 
the agents in $A$ according to the uniform distribution. 
Similar to the immediate exchange model, 
\cite{LR18} studied the case where $g$ is a constant function and 
$\rho$ is the uniform distribution on an edge set. 
In that model, $c(x)$ is drawn from the uniform distribution 
on $\{0, 1, 2, \cdots, \xi(x)\}$, and 
it is referred to as the {\em uniform saving model}. 
Since our model considers a more general distribution for $c(x)$, 
we refer to it as the {\em random saving model}. 
As a special case, when the weight function $g$ is defined 
by $g(k)=\delta_0(k)$ for $k\in \mathbb{Z}_+$, 
each agent saves no money. 
When a money exchange occurs within set $A$, 
all coins held by the agents in $A$ are pooled and redistributed. 
This model is referred to as the {\em uniform reshuffling model}, 
and we denote it by $\{Z_n\}_{n\geq 0}$. 

As the first result, we characterize the stationary distributions 
of these Markov chains. 
\begin{prop}\label{prop1}
Let $N, L \in \mathbb{N}$ be fixed. 
We assume that the hypergraph $(\Lambda_N, \mathcal{D}_{N, \rho})$ 
is connected where the hyperedge set 
$\mathcal{D}_{N, \rho}$ is defined by 
$\mathcal{D}_{N, \rho} = \{A\subset \Lambda_N; |A|\geq2, 
\rho(A)>0\}$.
\begin{enumerate}[$(i)$]
\item
Assume that the weight function $g:\mathbb{Z}_+ \to [0, \infty)$ 
satisfies $g(0)>0$ and $g(1)>0$. 
Then, there is a unique stationary distribution $\mu_{N, L}$ 
for $\{X_n\}_{n \geq 0}$ and it is given by 
\begin{equation}\label{stat}
\mu_{N, L}(\xi) = \frac{1}{Z_{N, L}}
\prod_{x\in \Lambda_N}\!\! G(\xi(x)), \ 
\xi \in \Omega_N(L), 
\end{equation}
where 
$G(k) = \sum\limits_{j=0}^k g(j),\ k \in \mathbb{Z}_+$ 
and 
$Z_{N, L}= \sum\limits_{\xi \in \Omega_N(L)}
\prod\limits_{x\in \Lambda_N} G(\xi(x))$ is the normalization factor. 
In particular, for every $\eta \in \Omega_N(L)$, 
$x\in \Lambda_N$ 
and $k \in \{0, 1, 2, \cdots, L\}$, 
it holds that 
\begin{align*}
\lim_{n \to \infty} P_{\eta}\bigl( X_n(x) = k \bigr) 
= \mu_{N, L}\bigl( \xi(x) =k \bigr) = 
\frac{Z_{N-1, L-k}\, G(k)}{Z_{N, L}}, 
\end{align*}
where $P_\eta$ represents the law of $\{X_n\}_{n\geq 0}$ 
with the initial condition $X_0=\eta$. 

\item
Assume that the weight function $g:\mathbb{Z}_+ \to [0, \infty)$ 
satisfies $g(0)>0$. 
Then, the exact same statement as $(i)$ applies to 
$\{Y_n\}_{n \geq 0}$ instead of $\{X_n\}_{n \geq 0}$. 
In particular, for the uniform reshuffling model $\{Z_n\}_{n\geq 0}$, 
there is a unique stationary distribution $\pi_{N, L}$ 
and it is given by the uniform distribution on $\Omega_N(L)$, 
namely, 
\begin{equation*}
\pi_{N, L}(\xi) = \frac{1}{|\Omega_{N}(L)|}, \ 
\xi \in \Omega_N(L). 
\end{equation*}
\end{enumerate}
\end{prop}
\noindent 
By this proposition we can see that both the immediate exchange model 
and the random saving model have the same stationary distribution 
under the condition $g(0)>0$ and $g(1)>0$.  
Also, the choice of $\rho$ does not affect the stationary distribution. 

\begin{remark}
The condition $g(0)>0$ is always needed to make the 
measure (\ref{stat}) well-defined. 
\end{remark}
\begin{remark}
Consider the case where the weight function $g$ is a constant function 
$g \equiv \gamma>0$. 
Then, it holds that 
$G(k)= \gamma(k+1)$, $k\in \mathbb{Z}_+$ and 
the constant $\gamma$ is canceled by the normalization 
factor in the definition of $\mu_{N, L}$. 
Therefore, we can take $G$ as 
$G(k)= k+1$, $k\in \mathbb{Z}_+$ in (\ref{stat}) 
and in this case, the above result matches that 
of \cite{LR18}. 
\cite{LR18} also gave the explicit representation for 
$Z_{N, L}$ when $g\equiv 1$ and obtained (\ref{lr1}). 
\end{remark}
\begin{remark}
For the random saving model 
we can change the role of random variable $c(x)$ 
to represent the number of coins to offer 
instead of the number of coins to save. 
Namely, we modify $(3)$ and $(4)$ in the definition of the 
random saving model as follows: 
\begin{enumerate}
\item[$(3)'$]
For given $\xi$ and $c=\{c(x)\}_{x\in A}$, 
choose a configuration $d =\{d(x)\}_{x\in A} 
\in \Omega(A, S_A(c))$ uniformly random, 
namely with probability 
$\frac{1}{|\Omega(A, S_A(c))|}$. 
\item[$(4)'$]
For given $\xi$ and realizations $A$, $\{c(x)\}_{x\in A}$ and 
$\{d(x)\}_{x\in A}$, define $Y_{n+1}$ by 
\begin{equation*}
Y_{n+1}(z) = 
\begin{cases}
\xi(z)-c(z) + d(z) & \text{ if } z \in A, \\
\xi(z) & \text{ if } z \notin A. \\
\end{cases}
\end{equation*}
\end{enumerate}
By symmetry, the completely same proof for Proposition \ref{prop1} 
below works well in this setting and the same result 
holds for this modified model. 
\end{remark}

Now, we are in the position to state the main result of this paper. 
To justify the limit 
$N\to \infty$ and $\frac{L}{N}\to \infty$, 
we assume that the total number of coins 
$L=L_N$ satisfies 
$\lim\limits_{N\to\infty}\frac{L_N}{Na_N}=T$ 
for some $T>0$ and divergent sequence 
$\{a_N\}_{N\geq 1}$. 
Then, we can prove that the law of the scaled field 
$\bigl\{ \frac{1}{a_N}\eta(x)\bigr\}_{x\in \Lambda_N}$ under 
$\mu_{N, L}$ or $\pi_{N, L}$ converges 
to the i.i.d. product of 
probability distributions on $\mathbb{R}_+$. 
Its marginal distribution depends on the 
asymptotic behavior of the weight function $g$. 
\begin{thm}\label{thm1}
Let $\{L_N \}_{N\geq 1}$ be a sequence of positive integers 
that satisfies 
$\lim\limits_{N\to\infty}\frac{L_N}{N a_N}=T$ 
for some positive constant $T>0$ and a sequence 
$\{a_N\}_{N\geq 1}$ which satisfies 
$\lim\limits_{N\to \infty}a_N=\infty$. 
Assume that $g(0)>0$ and the following condition holds: 
{There exist } $\alpha\in \mathbb{R}$ and 
$c_\alpha \in (0, \infty)$ such that 
$\lim\limits_{k\to \infty}\frac{g(k)}{k^\alpha} =c_\alpha$. 
Then, for 
every bounded continuous local function 
$f:\mathbb{R}^A \to \mathbb{R}$, it holds that 
\begin{align*}
\lim_{N\to \infty}
E^{\mu_{N, L_N}} \Bigl[ f\bigl( 
\frac{\cdot}{a_N} \bigr)\Bigr] 
= 
E^{\overline{\mu}_{\alpha, T}^A}\bigl[f(\, \cdot\, )\bigr], 
\end{align*}
where $A$ is a finite subset of $\mathbb{Z}_+$ and 
$\overline{\mu}_{\alpha, T}^A$ denotes 
the product probability measure on 
$\mathbb{R}_+^A$ whose one site marginal distribution 
on $\mathbb{R}_+$ is given by 
\begin{align}\label{limdist}
\mu_{\alpha, T}(dr) = 
\begin{cases}
\frac{1}{\Gamma(\alpha+2)} \bigl(\frac{\alpha+2}{T}
\bigr)^{\alpha +2} r^{\alpha +1} e^{-\frac{\alpha+2}
{T}r}dr & \text{ if } \alpha>-1, \\
\frac{1}{T} e^{-\frac{1}{T}r}dr 
& \text{ if } \alpha \leq -1. 
\end{cases}
\end{align}
$\Gamma(\beta)= \int_0^\infty r^{\beta-1}e^{-r}dr$ is 
the Gamma function with parameter $\beta>0$. 

Also, if $g: \mathbb{Z}_+\to [0, \infty)$ satisfies 
$g(0)>0$ and $\sum\limits_{j\geq 0} g(j) <\infty$, then 
the same conclusion as in the case $\alpha\leq -1$ above holds.
In particular, for every bounded continuous local function 
$f:\mathbb{R}^A \to \mathbb{R}$, it holds that 
\begin{align*}
\lim_{N\to \infty}
E^{\pi_{N, L_N }} \Bigl[ f\bigl( 
\frac{\cdot}{a_N} \bigr)\Bigr] 
= 
E^{{\overline{\pi}}_{T}^A}\bigl[f(\, \cdot\, )\bigr], 
\end{align*}
where $A$ is a finite subset of $\mathbb{Z}_+$ and 
$\overline{\pi}_{T}^A$ denotes 
the product probability measure on 
$\mathbb{R}_+^A$ whose one site marginal distribution 
on $\mathbb{R}_+$ is given by 
$\pi_{T}(dr) = 
\frac{1}{T} e^{-\frac{1}{T}r}dr$.  
\end{thm}
As an easy consequence of Proposition \ref{prop1} 
and Theorem \ref{thm1}, we obtain the following. 
\begin{cor}\label{cor1}
Let $\{X_n^{(N)}\}_{n\geq 0}$, $\{Y_n^{(N)}\}_{n\geq 0}$ 
and $\{Z_n^{(N)}\}_{n\geq 0}$ denote the immediate exchange model, 
the random saving model and the uniform reshuffling model 
on the state space $\Omega_{N}(L_N)$, respectively. 
Under the same conditions as in Proposition \ref{prop1} 
and Theorem \ref{thm1}, we have 
\begin{align*}
\lim_{N\to \infty}\lim_{n\to \infty}
E_{\eta} \Bigl[ \frac{1}{N}
\bigm|\!\!\bigl\{x\in \Lambda_N; 
\frac{1}{a_N} {X_n^{(N)} (x)} \in (b, c) \bigr\}
\!\!\bigm|\Bigr] 
= \mu_{\alpha, T}((b, c)), 
\end{align*}
\begin{align*}
\lim_{N\to \infty}\lim_{n\to \infty}
E_{\eta} \Bigl[ \frac{1}{N}
\bigm|\!\!\bigl\{x\in \Lambda_N; 
\frac{1}{a_N} {Y_n^{(N)} (x)} \in (b, c) \bigr\}
\!\!\bigm|\Bigr] 
= \mu_{\alpha, T}((b, c)), 
\end{align*}
and 
\begin{align*}
\lim_{N\to \infty}\lim_{n\to \infty}
E_{\eta} \Bigl[ \frac{1}{N}
\bigm|\!\!\bigl\{x\in \Lambda_N; 
\frac{1}{a_N} {Z_n^{(N)} (x)} \in (b, c) \bigr\}
\!\!\bigm|\Bigr] 
= \pi_{T}((b, c)), 
\end{align*}
for every $\eta \in \Omega_N(L_N)$ and 
every $0 \leq b < c \leq \infty$ 
where $E_{\eta}[\ \cdot \ ]$ denotes the expectation 
with respect to the law of the Markov chain with the 
initial condition $\eta$. 
\end{cor}
\noindent
We note that it is natural to consider the 
scaling of the process by a factor of 
$\frac{1}{a_N}$. 
This is because, under the condition on $L_N$, 
we have 
$\frac{1}{N}\sum\limits_{x\in \Lambda_N} \bigl(
\frac{1}{a_N}X_n^{(N)} (x)
\bigr) = \frac{L_N}{Na_N}= T(1+o(1))$ as 
$N\to \infty$, 
which means that the asymptotic average number of coins per agent 
for the scaled process is given by $T$. 
This value corresponds to the money temperature 
in our model. 
Corollary \ref{cor1} provides a precise formulation and generalization 
of earlier studies in the physics literature, 
which were based on numerical simulations and heuristic arguments. 
As time approaches infinity, and in the large population and large 
money temperature limit, the asymptotic wealth distribution, 
i.e. the proportion of agents holding a specific number of coins 
converges to a Gamma distribution or an exponential distribution 
for both the immediate exchange model and 
the random saving model, while it converges to an exponential 
distribution for the uniform reshuffling model. 
The parameters of the Gamma distribution 
depend on the asymptotic behavior of the weight function $g$. 
When $\alpha>-1$ the limiting wealth distribution 
is given by a Gamma distribution with mean $T$ and 
shape parameter $\alpha +2$. 
While, when $\alpha \leq -1$, the limiting wealth distribution 
is given by an exponential distribution with mean $T$ and 
this does not depend on the parameter $\alpha$. 
In particular, if $g$ is a constant function 
then $\alpha=0$ and the limiting distribution 
is a Gamma distribution with mean $T$ and 
shape parameter two, which corresponds to 
(\ref{app}). 
Additionally, when $g(k)=(k+1)^{\alpha}$, 
$k\in \mathbb{Z}_+$, 
the above results are consistent with numerical simulations; 
see Figures 1 and 2. 
\begin{figure}[t]
\begin{tabular}{cc}
\begin{minipage}[t]{0.5\hsize}
\centering
\includegraphics[keepaspectratio, scale=0.5]{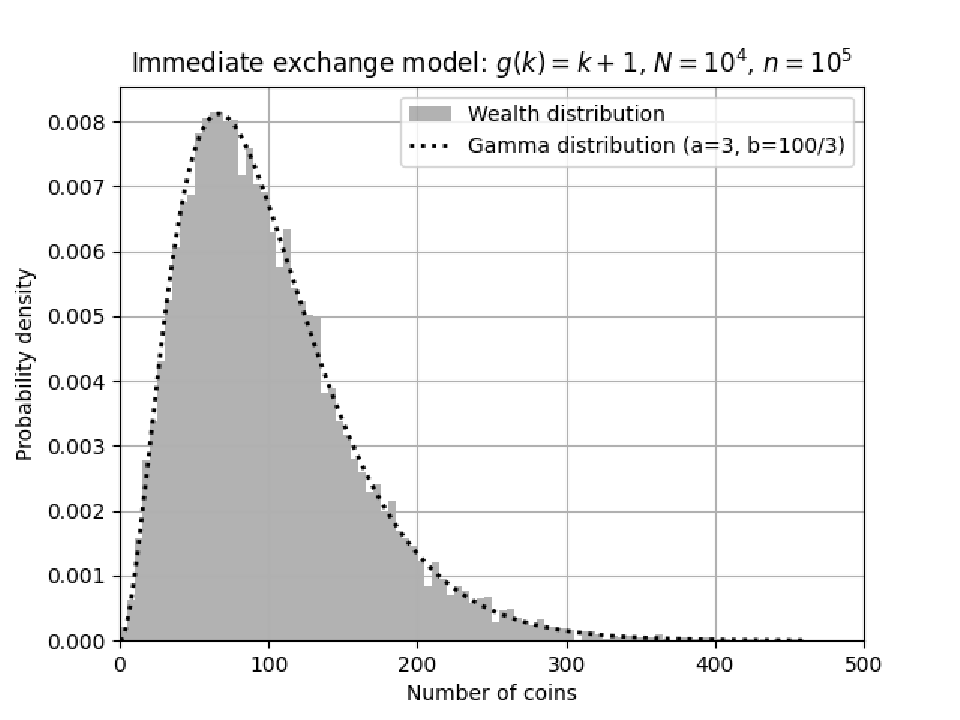}
\end{minipage} &
\begin{minipage}[t]{0.5\hsize}
\includegraphics[keepaspectratio, scale=0.5]{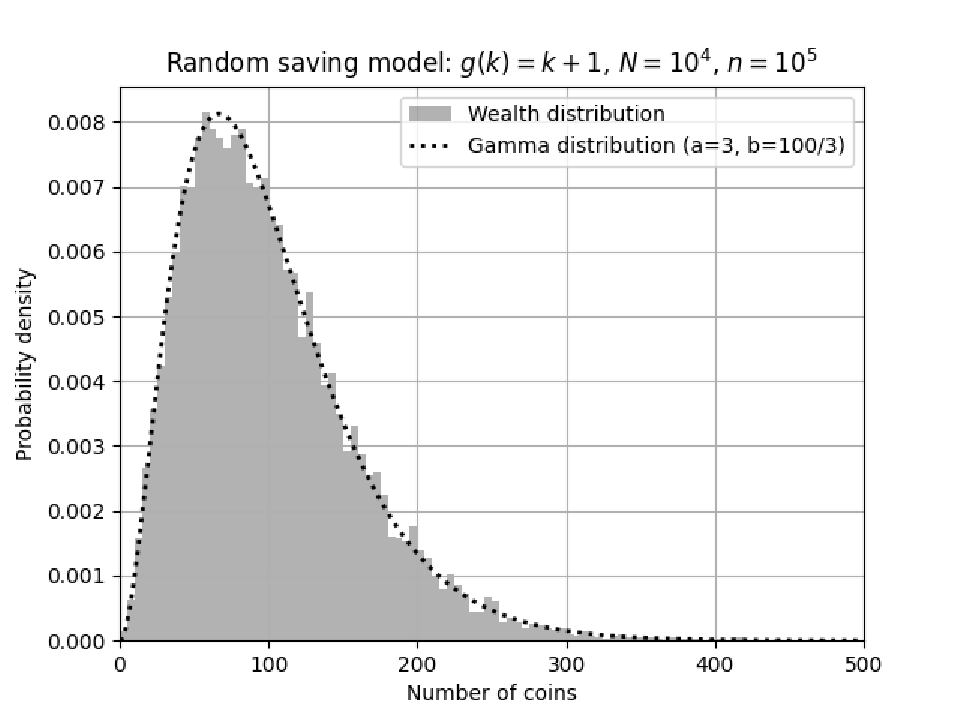}        
\end{minipage} \\
\begin{minipage}[t]{0.5\hsize}
\centering
\includegraphics[keepaspectratio, scale=0.5]{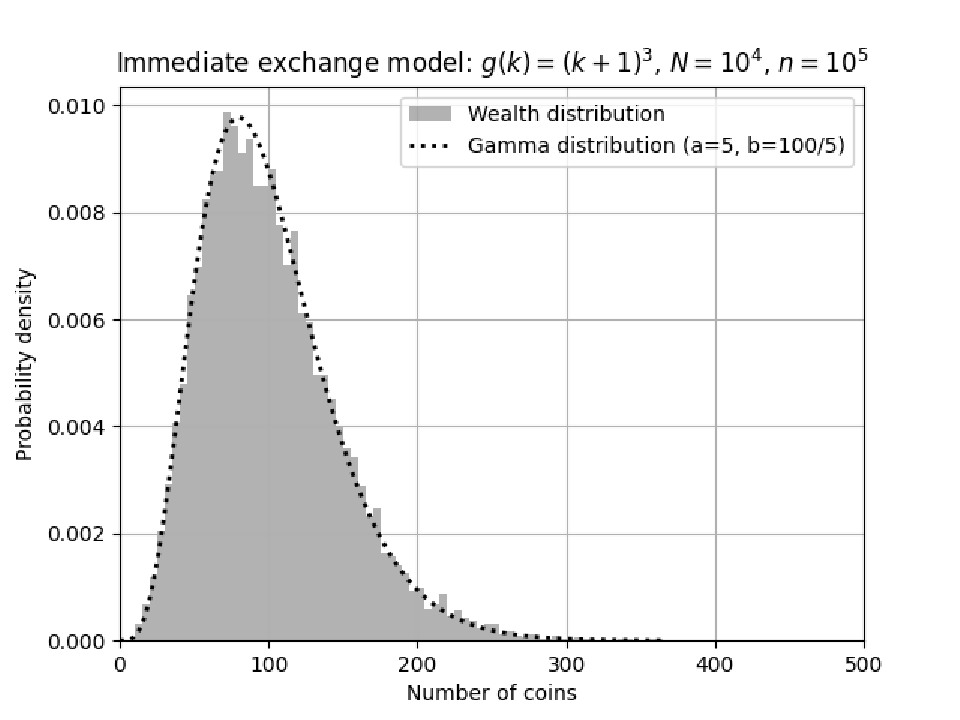}        
\end{minipage} &
\begin{minipage}[t]{0.5\hsize}
\includegraphics[keepaspectratio, scale=0.5]{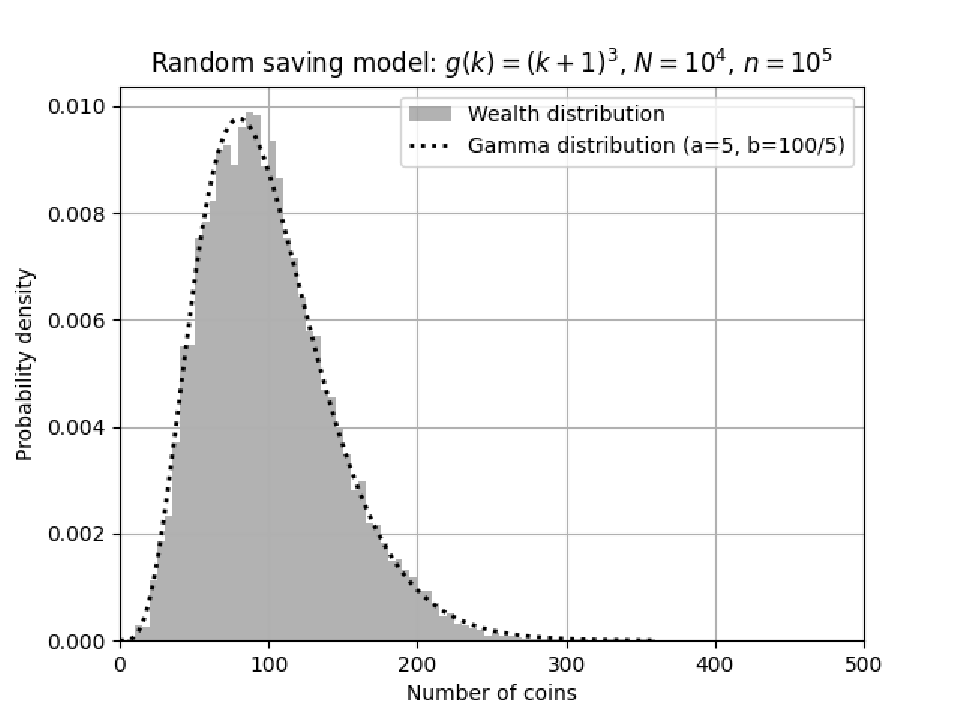}        
\end{minipage} 
\end{tabular}
\caption{Simulation results for a 
single realization of the immediate exchange model and the 
random saving model with $g(k)=(k+1)^{\alpha}$ where 
$\alpha = 1, 3$. 
The number of agents is $N=10^4$ and the total number of 
coins is $L=10^6$, namely the average number of coins per 
agent equals to $100$. 
The initial condition is set to a constant configuration 
$X_0 \equiv 100$ or $Y_0 \equiv 100$. 
$\rho$ is distributed uniformly over the edge set 
$\{\{x, y\}; x, y\in \Lambda_N, x\ne y\}$, 
and in the immediate exchange model, 
swapping shall always be performed between the selected edges.
The gray histograms represent 
the wealth distribution, i.e. 
the proportion of agents holding a specific number of coins 
after $n=10^5$ updates. 
The dotted line is the graph of the 
probability density function of the Gamma distribution: 
$f_{a, b}(r)= 
\frac{1}{\Gamma(a)b^{a}} r^{ a-1} 
e^{-\frac{1}{b}r}$ 
with the shape parameter $a=\alpha+2$ 
and the scale parameter $b= \frac{100}{\alpha+2}$. 
}
\end{figure}
\begin{figure}[htbp]
\begin{tabular}{cc}
\begin{minipage}[t]{0.5\hsize}
\centering
\includegraphics[keepaspectratio, scale=0.5]{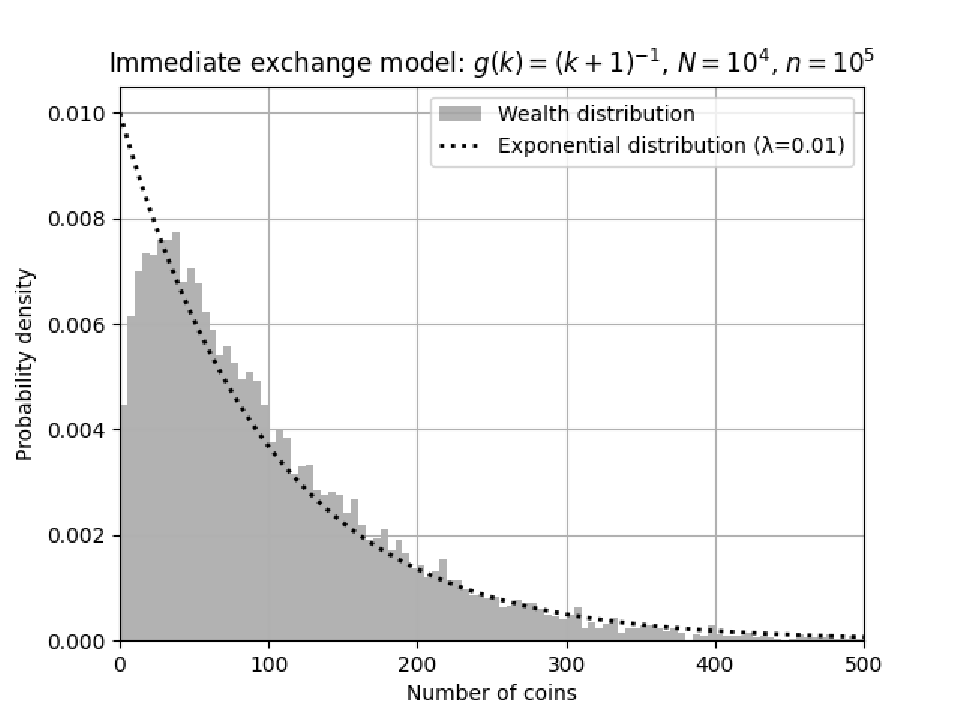}
\end{minipage} &
\begin{minipage}[t]{0.5\hsize}
\includegraphics[keepaspectratio, scale=0.5]{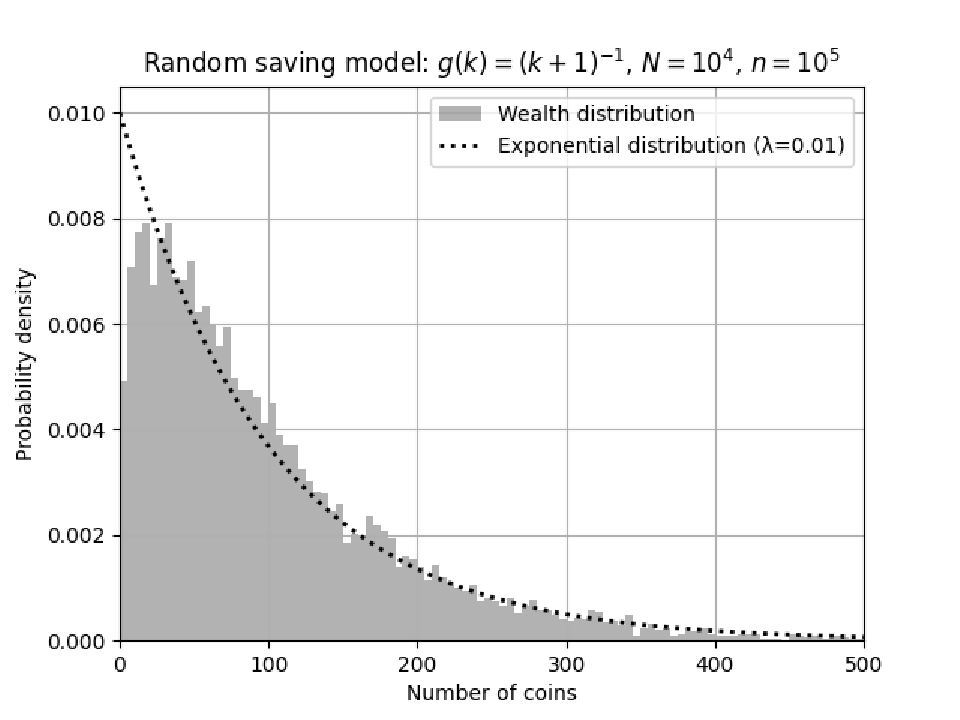}        
\end{minipage} \\
\begin{minipage}[t]{0.5\hsize}
\centering
\includegraphics[keepaspectratio, scale=0.5]{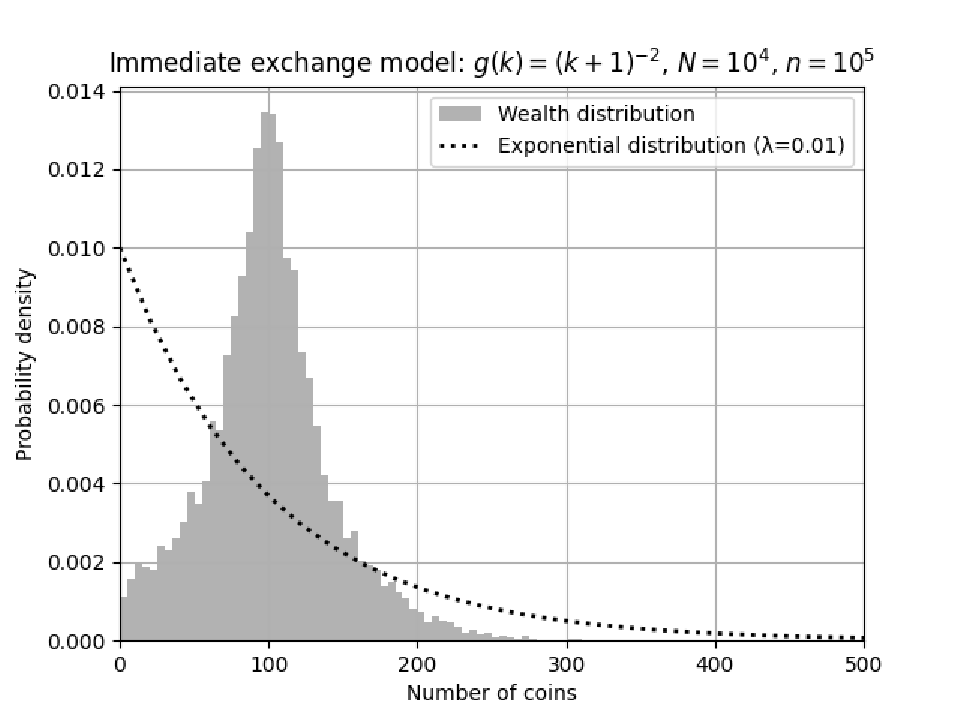}        
\end{minipage} &
\begin{minipage}[t]{0.5\hsize}
\includegraphics[keepaspectratio, scale=0.5]{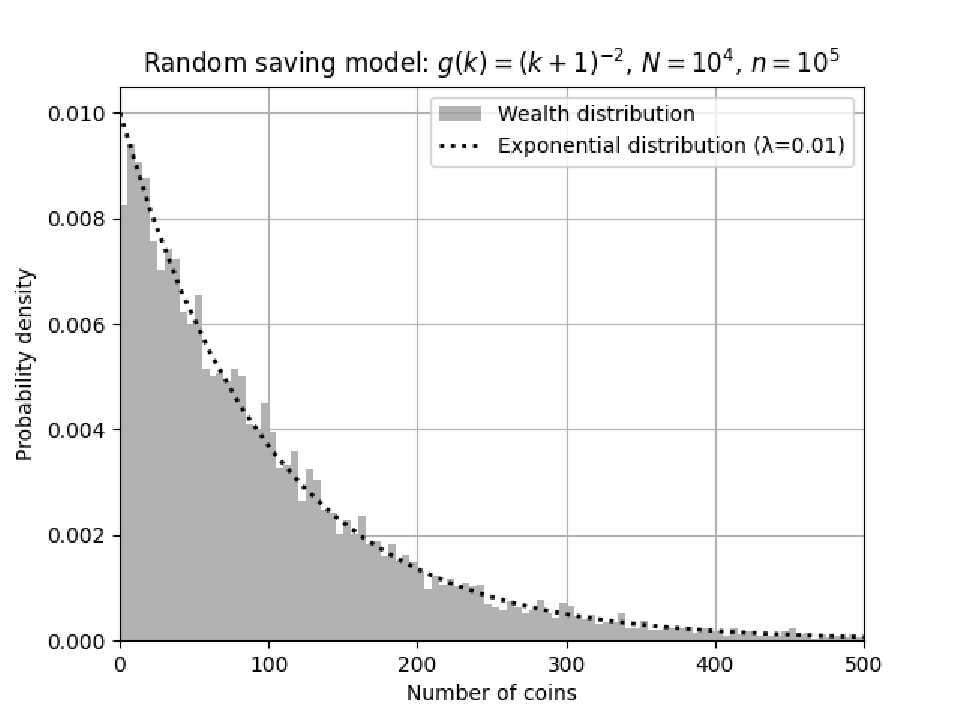}        
\end{minipage} \\
\begin{minipage}[t]{0.5\hsize}
\centering
\includegraphics[keepaspectratio, scale=0.5]{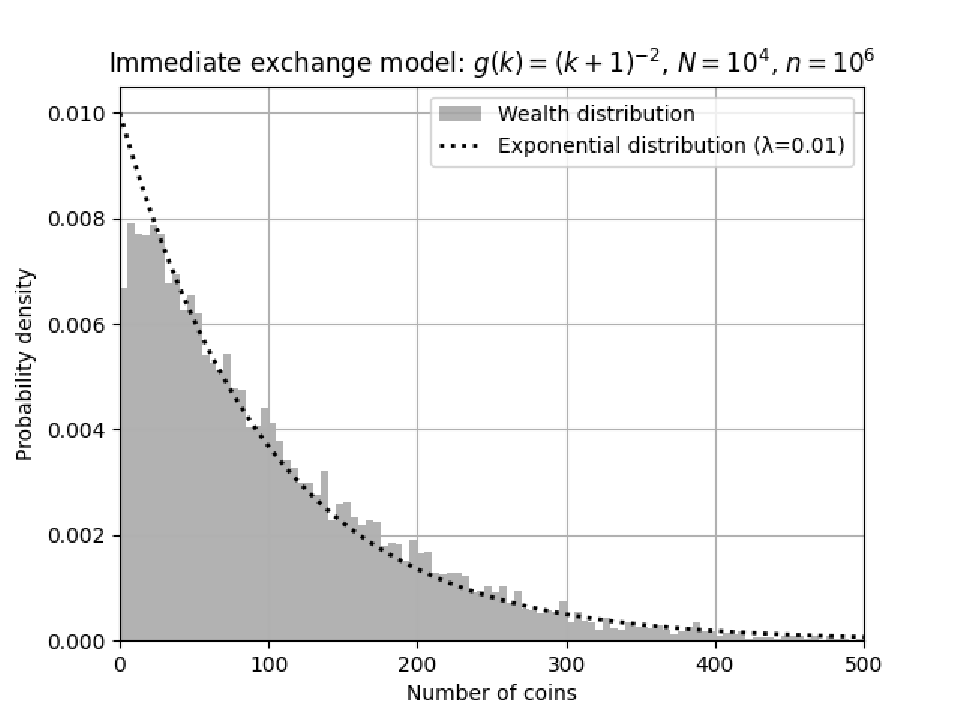}        
\end{minipage} &
\begin{minipage}[t]{0.5\hsize}
\includegraphics[keepaspectratio, scale=0.5]{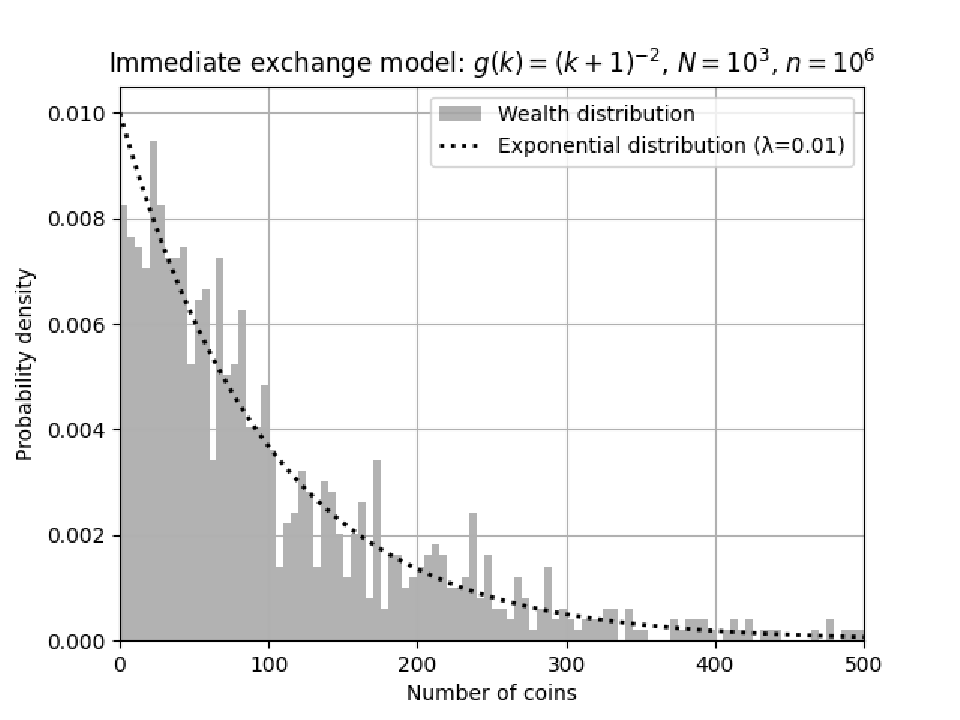}        
\end{minipage} 
\end{tabular}
\caption{Simulation results for a 
single realization of the immediate exchange model and the 
random saving model with $g(k)=(k+1)^{\alpha}$ where 
$\alpha = -1, -2$. 
The settings of the simulations are the same as Figure 1. 
The dotted line is the graph of the 
probability density function of the exponential distribution: 
$f_{\lambda}(r)= {\lambda} 
e^{-\lambda r}$ 
with the parameter $\lambda=\frac{1}{100}$. 
When the weight function $g$ decays rapidly 
in the immediate exchange model, 
the probability of each agent exchanging 
only a small number of coins at a time increases, 
leading to a longer convergence time to reach a steady state. 
The simulation result for the case $g(k)=(k+1)^{-2}$ 
reflects this and 
$n=10^5$ updates are not sufficient for convergence 
(the middle left).
However, after $n=10^6$ updates, the histograms approach 
the limiting probability density function (the lower left). 
Also, the result for the case 
$N=10^3$ and $n=10^6$ indicates that 
we need to take a large population $N$ 
(and hence a large $L$) for convergence (the lower right). 
}
\end{figure}
\begin{remark}
The limiting distribution (\ref{limdist}) for the case $\alpha \leq -1$ 
can be interpreted as follows: 
If the weight function $g$ decays rapidly in the random saving model, 
the probability that each agent saves a large number of coins becomes 
very small. 
Moreover, since we are considering the scaled process 
$\{\frac{1}{a_N}Y_n^{(N)}\}_{n\geq 0}$, 
we can assume that each agent offers nearly all the coins they have in 
each money exchange. 
Consequently, similarly to the uniform reshuffling model, 
the limiting distribution becomes an exponential 
distribution with mean $T$. 
In the immediate exchange model, 
when the weight function $g$ decays rapidly, 
the probability that each agent exchanges a large number of 
coins also becomes very small. 
For the scaled process, 
this situation can be regarded as similar to the one-coin model, 
where one agent gives only one coin to another agent at a time. 
In the one-coin model, the limiting distribution is expected to be exponential 
(cf. \cite{DY00}, \cite{L17}) and the result above 
aligns with this. 
Furthermore, a mathematically rigorous justification of the result 
for the one-coin model can be achieved 
by formulating and proving it in the same manner as demonstrated in this paper.
\end{remark}

In the rest of the paper 
we provide the proofs of Proposition \ref{prop1} 
and Theorem \ref{thm1} 
in Sections 2 and 3, respectively. 
We give some comments about the strategy of the proof. 
The proof of Proposition \ref{prop1} is standard. 
It is not difficult to see that our models are 
irreducible aperiodic 
Markov chains on the finite state space $\Omega_N(L)$. 
All we have to do is to characterize the unique stationary distribution, 
which is achieved by carefully verifying the detailed balance condition.
Namely, we demonstrate that our models are reversible Markov chains, 
and the reversible distributions for
$\{X_n\}_{n\geq 0}$ and $\{Y_n\}_{n\geq 0}$ are 
given by (\ref{stat}). 
With respect to the proof of Theorem \ref{thm1}, 
the convergence of the marginal distribution of $\mu_{N, L}$ 
is closely related to 
the equivalence of ensembles (cf. \cite{FV}, \cite{KL}). 
The sequence $\{G(k)\}_{k\in \mathbb{Z}_+}$ 
is neither a probability distribution nor generally summable over $k$. 
However, multiplying $G(k)$ by the exponential factor 
$s^k$, $s\in [0, 1)$ ensures the convergence of 
$\sum\limits_{k\in \mathbb{Z}_+}s^kG(k)$ under the assumption on $g$. 
The stationary distribution (\ref{stat}) can then be interpreted as 
the microcanonical distribution of an i.i.d. product, 
where the one-site marginal distribution is of exponential type 
and proportional to $G$. 
Instead of the usual condition 
$\frac{1}{N}\sum\limits_{x\in \Lambda_N}\eta(x) \to m \in (0, \infty)$ 
as $N\to \infty$, 
we consider the condition 
$\frac{1}{Na_N}\sum\limits_{x\in \Lambda_N}\eta(x) \to T \in (0, \infty)$ 
and investigate the convergence of the law of the scaled field 
$\frac{1}{a_N}\eta$ under the corresponding microcanonical distribution. 
We adapt the proof of the equivalence of ensembles from 
\cite[Appendix 2]{KL} to this unusual setting.
In particular, the local limit theorem 
for a triangular array of random variables 
plays an important role in the argument. 

Throughout the paper $C$, $C'$, $C''$ represent positive constants 
that do not depend on the size of the system $N$, 
but may depend on other parameters. 
These constants in various 
estimates may change from place to place in the paper. 


\section{Proof of Proposition \ref{prop1}}
In this section we prove Proposition \ref{prop1}. 
Under the condition that the hypergraph $(\Lambda_N, \mathcal{D}_{N, \rho})$ 
is connected, we have that 
for every $x, y\in \Lambda_N$ there exists a sequence 
$\{A_k\}_{0 \leq k \leq l}\subset \mathcal{D}_{N, \rho}$ 
such that 
$x\in A_0$, $y\in A_l$ and $A_{k} \cap A_{k+1} 
\ne \emptyset$ for every $0 \leq k \leq l-1$. 
Combining this fact with the assumption $g(0)>0$ 
and $g(1)>0$, 
it is easy to see that the following holds. 
\begin{itemize}
\item
For every $\xi, \eta\in \Omega_N(L)$ there exists $m=m(\xi, \eta)\geq 0$ 
such that 
$P(X_{m} = \eta | X_0=\xi) >0$. 
\item
$P(X_{n+1} = \eta | X_n=\eta) >0$ for every $\eta \in \Omega_N(L)$ 
and $n\geq 0$. 
\end{itemize}
The same statement also holds for $\{Y_n\}_{n\geq 0}$ 
under the condition $g(0)>0$. 
Namely, $\{X_n\}_{n \geq 0}$ and $\{Y_n\}_{n \geq 0}$ 
are irreducible aperiodic 
Markov chains on the finite state space 
$\Omega_N(L)$. 
Then, the stationary distribution of 
$\{X_n\}_{n \geq 0}\ (\{Y_n\}_{n \geq 0})$ 
uniquely exists and the law of $X_n\ (Y_n)$ 
converges to it in the limit $n \to \infty$ 
by the Markov chain convergence theorem. 
Therefore, all we have to show is that 
the measure $\mu_{N, L}$ given by (\ref{stat}) is the 
stationary distribution for both of 
$\{X_n\}_{n \geq 0}$ and $\{Y_n\}_{n \geq 0}$. 
Actually, by verifying the detailed balance condition: 
\begin{equation*}
\mu_{N , L}(\xi) P(X_{n+1} = \eta | X_n=\xi) 
= 
\mu_{N , L}(\eta) P(X_{n+1} = \xi | X_n=\eta) \ 
\text{ for every } \xi, \eta\in \Omega_N(L), 
\end{equation*}
and 
\begin{equation*}
\mu_{N , L}(\xi) P(Y_{n+1} = \eta | Y_n=\xi) 
= 
\mu_{N , L}(\eta) P(Y_{n+1} = \xi | Y_n=\eta) \ 
\text{ for every } \xi, \eta\in \Omega_N(L), 
\end{equation*}
we prove that 
(\ref{stat}) is the reversible distribution 
for both of 
$\{X_n\}_{n \geq 0}$ and $\{Y_n\}_{n \geq 0}$. 

\medskip

\noindent 
{\em{Proof for the immediate exchange model:}}\ 
Take arbitrary $\xi, \eta \in \Omega_N(L)$. 
There exists $A_0=A_0(\xi, \eta) \subset \Lambda_N$ 
such that 
$\xi \ne \eta$ on $A_0$ and 
$\xi = \eta$ on $\Lambda_N\setminus A_0$. 
Such set $A_0$ is uniquely determined from $\xi$ and $\eta$. 
We have that 
\begin{align*}
P(X_{n+1} = \eta | X_n=\xi) 
& = 
\sum_{{\substack{A\in \mathcal{D}_N \\ A\supset A_0}}} 
\rho(A) \sum_{\sigma \in \mathcal{S}_A}
\frac{1}{|\mathcal{S}_A|} 
P\bigl(\eta(z)= \xi(z)-c(z) + c(\sigma^{-1}(z)) \text{ for every }
z\in A \bigm|\! \xi \bigr), 
\end{align*}
where $P(\ \cdot \ | \xi )$ denotes the law of 
$\{c(x)\}_{}$ 
in the dynamics $(2)$ for given the configuration $\xi$. 
If $\rho(A)=0$ for every $A\in \mathcal{D}_N$ so that $A\supset A_0$, 
then the right-hand side is equal to $0$. 
For a finite set $A$, we label its elements as 
$A=\{x_1, x_2, \cdots, x_{|A|}\}$. 
Then, 
\begin{align*}
& P\bigl(\eta(z)= \xi(z)-c(z) + c(\sigma^{-1}(z)) \text{ for every }
z\in A \bigm|\! \xi \bigr) \\
& = 
\sum_{t_1=0}^{\xi(x_1)}\sum_{t_2=0}^{\xi(x_2)}\cdots 
\sum_{t_{|A|}=0}^{\xi(x_{|A|})}
\prod_{i=1}^{|A|} \frac{g(t_i)}{G(\xi(x_i))} \cdot 
I\bigl( \eta(x_i)= \xi(x_i)-t_i + t_{\sigma^{-1}(i)}
\text{ for every } 1 \leq i \leq |A|
\bigr) \\
& = 
\sum_{t_1=0}^{\xi(x_1)}\sum_{t_2=0}^{\xi(x_2)}\cdots 
\sum_{t_{|A|}=0}^{\xi(x_{|A|})}
\prod_{i=1}^{|A|} \Bigl\{ 
\frac{g(t_i)}{G(\xi(x_i))} 
I\bigl( \eta(x_i)= \xi(x_i)-t_i + t_{\sigma^{-1}(i)} 
\bigr)\Bigr\}, 
\end{align*}
where 
for each $\sigma \in \mathcal{S}_A$ and $1\leq i\leq |A|$, 
we identify $\sigma^{-1}(i)$ with 
the label $1\leq k \leq |A|$ which satisfies 
$x_k= \sigma^{-1}(x_i)$. 
Therefore, 
\begin{align}\label{DB1}
\begin{split}
&
\Bigl\{ \prod_{x \in \Lambda_N} G(\xi(x)) \Bigr\} \cdot 
P(X_{n+1} = \eta | X_n=\xi) \\
& = 
\sum_{{\substack{A\in \mathcal{D}_N \\ A\supset A_0}}} 
\Bigl[ 
\prod_{x \in \Lambda_N\setminus A}\!\! 
G(\xi(x)) \cdot 
\frac{\rho(A)}{|\mathcal{S}_A|} 
\sum_{\sigma \in \mathcal{S}_A}
\sum_{t_1=0}^{\xi(x_1)}\sum_{t_2=0}^{\xi(x_2)}\cdots 
\sum_{t_{|A|}=0}^{\xi(x_{|A|})}
\prod_{i=1}^{|A|} \Bigl\{ 
{g(t_i)}
I\bigl( \eta(x_i)= \xi(x_i)-t_i + t_{\sigma^{-1}(i)} 
\bigr)\Bigr\}\Bigr] \\
& = 
\sum_{{\substack{A\in \mathcal{D}_N \\ A\supset A_0}}} 
\Bigl[
\prod_{x \in \Lambda_N\setminus A}\!\!
\Bigl\{
\frac{G(\xi(x))+G(\eta(x))}{2} 
\Bigr\} \cdot 
\frac{\rho(A)}{|\mathcal{S}_A|} 
\sum_{\sigma \in \mathcal{S}_A}
\sum_{t_1=0}^{\xi(x_1)}\sum_{t_2=0}^{\xi(x_2)}\cdots 
\sum_{t_{|A|}=0}^{\xi(x_{|A|})}
\prod_{i=1}^{|A|} \Bigl\{ 
{g(t_i)}
I\bigl( \eta(x_i)= \xi(x_i)-t_i + t_{\sigma^{-1}(i)} 
\bigr)\Bigr\}\Bigr],
\end{split} 
\end{align}
where the last equality follows from 
$\xi = \eta$ on $\Lambda_N\setminus A_0$. 
For $A\subset \Lambda_N$ and 
$\sigma \in \mathcal{S}_A$, we set 
\begin{equation*}
h_A(\sigma):=
\sum_{t_1=0}^{\xi(x_1)}\sum_{t_2=0}^{\xi(x_2)}\cdots 
\sum_{t_{|A|}=0}^{\xi(x_{|A|})}
\prod_{i=1}^{|A|} \Bigl\{ 
{g(t_i)}
I\bigl( \eta(x_i)= \xi(x_i)-t_i + t_{\sigma^{-1}(i)} 
\bigr)\Bigr\}. 
\end{equation*}
To show that (\ref{DB1}) is symmetric 
with respect to $\xi$ and $\eta$, it is sufficient to show 
that $\sum\limits_{\sigma \in \mathcal{S}_A}h_A(\sigma)$ 
is symmetric with respect to $\xi$ and $\eta$ 
for every $A\in \mathcal{D}_N$ so that 
$A\supset A_0$. 
Now, each permutation $\sigma \in\mathcal{S}_A$ 
can be decomposed as the product of cyclic permutations 
and the summand in $h_A(\sigma)$ is given by a product form. 
Accordingly, 
the following two statements are sufficient for the symmetry of 
$\sum\limits_{\sigma \in \mathcal{S}_A}h_A(\sigma)$.  
\begin{itemize}
\item
When $|A|=2$, $h_A(\sigma)$ is symmetric with respect to 
$\{\xi(x)\}_{x \in A}$ and 
$\{\eta(x)\}_{x \in A}$ for
every $\sigma \in\mathcal{S}_A$. 
\item
When $|A|\geq 3$, $h_A(\sigma)+ h_A(\sigma^{-1})$ is symmetric with respect to 
$\{\xi(x)\}_{x \in A}$ and 
$\{\eta(x)\}_{x \in A}$ for
every cyclic permutation $\sigma \in\mathcal{S}_A$. 
\end{itemize}
These can be reformulated as follows:
\begin{lem}
Let $g: \mathbb{Z}_+ \to [0, \infty)$ and 
$a= \{a_i\}_{i=1}^n, b= \{b_i\}_{i=1}^n \in 
\mathbb{Z}_+^n$ be two sequences of non-negative integers 
which satisfy $\sum\limits_{i=1}^n a_i = 
\sum\limits_{i=1}^n b_i$. 
\begin{enumerate}[$(i)$]
\item
Let $n=2$ and define 
\begin{align*}
S=S(a, b):=
\sum_{t_1=0}^{a_1}\sum_{t_2=0}^{a_2}
\prod_{i=1}^{2} \Bigl\{ 
{g(t_i)}
I\bigl( t_i - t_{i+1} = a_i-b_i \bigr) \Bigr\}, 
\end{align*}
where we identify $t_3$ as $t_1$. 
Then, $S$ is symmetric with respect to $a$ and $b$. 
\item
Let $n\geq 3$ and define 
\begin{align*}
S=S_+(a, b) + S_-(a, b)
& := 
\sum_{t_1=0}^{a_1}\sum_{t_2=0}^{a_2}
\cdots \sum_{t_n=0}^{a_n} 
\prod_{i=1}^{n} \Bigl\{ 
{g(t_i)}
I\bigl( t_i - t_{i+1} = a_i-b_i \bigr)\Bigr\} \\ 
& \qquad \qquad 
+ 
\sum_{t_1=0}^{a_1}\sum_{t_2=0}^{a_2}
\cdots \sum_{t_n=0}^{a_n} 
\prod_{i=1}^{n} \Bigl\{ 
{g(t_i)}
I\bigl( t_i - t_{i-1} = a_i-b_i \bigr)\Bigr\}, 
\end{align*}
where we identify $t_{n+1}$ as $t_1$ and $t_0$ as $t_n$. 
Then, $S$ is symmetric with respect to $a$ and $b$. 
\end{enumerate}
\end{lem}
\begin{proof}
\begin{enumerate}[$(i)$]
\item
Under the condition $a_1+a_2= b_1+b_2$, we have 
\begin{align*}
S(a, b) 
& =
\sum_{t_1=0}^{a_1}\sum_{t_2=0}^{a_2}
g(t_1)g(t_1-a_1+b_1)
I\bigl( t_{2} = t_1 -a_1+b_1 \bigr) \\
& = 
\sum_{t_1=0}^{a_1}
g(t_1)g(t_1-a_1+b_1)
I\bigl( 0 \leq  t_1 -a_1+b_1 \leq a_2 \bigr) \\
& = 
\sum_{t_1=0 \vee (a_1-b_1)}^{a_1\wedge b_2}
g(t_1)g(t_1-a_1+b_1) \\
& = 
\sum_{t_1= a_1  \vee b_1}^{(a_1+b_1) \wedge (b_1+b_2)}
g(t_1-b_1 )g(t_1-a_1). 
\end{align*}
\noindent
This is symmetric with respect to $a$ and $b$ 
because 
$b_1+b_2= \frac{1}{2}(
a_1+a_2+b_1+b_2)$. 
\item
We compute that 
\begin{align*}
S_+(a, b) 
& = 
\sum_{t_1 \geq 0}\sum_{t_2 \geq 0}
\cdots \sum_{t_n\geq 0} 
\Bigl[\prod_{i=1}^{n} \Bigl\{ 
{g(t_i)}
I\bigl( t_i - t_{i+1} = a_i-b_i \bigr)\Bigr\} 
\prod_{i=1}^{n} 
\Bigl\{I\bigl( t_i \leq a_i) I( t_{i+1} \leq b_i \bigr)
\Bigr\}\Bigr] \\ 
& = 
\sum_{t_1 \geq 0}\sum_{t_2 \geq 0}
\cdots \sum_{t_n\geq 0} 
\prod_{i=1}^{n} 
{g(t_i)}
\prod_{i=1}^{n} \Bigl\{ 
I\bigl( t_i - t_{i+1} = a_i-b_i \bigr)
I\bigl( t_i \leq a_i\wedge b_{i-1} \bigr)\Bigr\}. 
\end{align*}
Similarly, 
\begin{align*}
S_-(a, b) 
& = 
\sum_{t_1 \geq 0}\sum_{t_2 \geq 0}
\cdots \sum_{t_n\geq 0} 
\prod_{i=1}^{n} 
{g(t_i)}
\prod_{i=1}^{n} \Bigl\{ 
I\bigl( t_i - t_{i-1} = a_i-b_i \bigr)
I\bigl( t_i \leq a_i\wedge b_{i+1} \bigr)\Bigr\} \\
& = 
\sum_{t_1 \geq 0}\sum_{t_2 \geq 0}
\cdots \sum_{t_n\geq 0} 
\prod_{i=1}^{n} 
{g(t_i)}
\prod_{i=1}^{n} \Bigl\{ 
I\bigl( t_{i+1} - t_{i} = a_i-b_i \bigr)
I\bigl( t_{i+1} \leq a_i\wedge b_{i+1} \bigr)\Bigr\} \\
& = 
\sum_{t_1 \geq 0}\sum_{t_2 \geq 0}
\cdots \sum_{t_n\geq 0} 
\prod_{i=1}^{n} 
{g(t_i)}
\prod_{i=1}^{n} \Bigl\{ 
I\bigl( t_{i} - t_{i+1} = b_i-a_i \bigr)
I\bigl( t_{i} \leq a_{i-1} \wedge b_{i} \bigr)\Bigr\}\\
& = 
S_+(b, a), 
\end{align*}
where the second equality follows from 
rewriting the variable $t_i$ by $t_{i+1}$, $1\leq i \leq n$.  
Therefore, 
$S_+(a, b)+S_-(a, b)$ 
is symmetric with respect to $a$ and $b$. 
\end{enumerate}

\end{proof}

\medskip

\noindent
{\em{Proof for the random saving model:}}\ 
Take arbitrary $\xi, \eta \in \Omega_N(L)$. 
We use the similar notation as the 
proof for the immediate exchange model. 
\begin{align*}
& 
P(Y_{n+1} = \eta | Y_n=\xi) \\
& = 
\sum_{{\substack{A\in \mathcal{D}_N \\ A\supset A_0}}} 
\rho(A) 
P\bigl(\eta(z)= c(z) + d(z) \text{ for every }
z\in A \bigm| \xi \bigr) \\
& = 
\sum_{{\substack{A\in \mathcal{D}_N \\ A\supset A_0}}} 
\rho(A) 
\sum_{t_1=0}^{\xi(x_1)}\sum_{t_2=0}^{\xi(x_2)}\cdots 
\sum_{t_{|A|}=0}^{\xi(x_{|A|})}
\Bigl[
\prod_{i=1}^{|A|} \frac{g(t_i)}{G(\xi(x_i))} 
\sum\limits_{\zeta \in 
\Omega(A, S_A(\xi)-S_A(t))} 
\frac{1}{|\Omega(A, S_A(\xi)-S_A(t))|} \\
& \qquad \qquad \qquad \qquad \qquad \qquad \qquad \qquad 
\times 
I\bigl( \eta(x_i)= t_i + \zeta(x_i)
\text{ for every } 1 \leq i \leq |A|
\bigr) \Bigr] \\
& = 
\sum_{{\substack{A\in \mathcal{D}_N \\ A\supset A_0}}} 
\rho(A) 
\sum_{t_1=0}^{\xi(x_1)\wedge \eta(x_1)}
\sum_{t_2=0}^{\xi(x_2)\wedge \eta(x_2)}\cdots 
\sum_{t_{|A|}=0}^{\xi(x_{|A|})\wedge \eta(x_{|A|})}
\Bigl[
\prod_{i=1}^{|A|} \frac{g(t_i)}{G(\xi(x_i))} 
\sum\limits_{\zeta \in 
\Omega(A, S_A(\xi)-S_A(t))} 
\frac{1}{|\Omega(A, S_A(\xi)-S_A(t))|} \\
& \qquad \qquad \qquad \qquad \qquad \qquad \qquad \qquad 
\times 
I\bigl( \eta(x_i)= t_i + \zeta(x_i)
\text{ for every } 1 \leq i \leq |A|
\bigr)\Bigr], 
\end{align*}
where $S_A(t)=\sum\limits_{i=1}^{|A|}t_i$ for 
$t=\{t_i\}_{i=1}^{|A|}$. 
We have $S_A(\xi)=S_A(\eta)$ for $A\supset A_0$. 
Also, 
for given $t=\{t_i\}_{i=1}^{|A|}$ so that 
$0 \leq t_i \leq \xi(x_i) \wedge \eta(x_i)$ for 
every $1\leq i \leq |A|$, there exists unique 
$\zeta \in \Omega(A, S_A(\xi)-S_A(t))$ which 
satisfies $\eta(x_i)= t_i + \zeta(x_i)$ 
for every $1 \leq i \leq |A|$. 
Therefore, 
\begin{align*}\label{DB1}
& 
\Bigl\{ \prod_{x \in \Lambda_N} G(\xi(x))\Bigr\} \cdot 
P(Y_{n+1} = \eta | Y_n=\xi) \\
& \quad = 
\sum_{{\substack{A\in \mathcal{D}_N \\ A\supset A_0}}} 
\Bigl[
\prod_{x \in \Lambda_N\setminus A}\!\! \Bigl\{
\frac{G(\xi(x))+G(\eta(x))}{2} 
\Bigr\} \cdot 
{\rho(A)}
\sum_{t_1=0}^{\xi(x_1)\wedge \eta(x_1)}
\sum_{t_2=0}^{\xi(x_2)\wedge \eta(x_2)}\cdots 
\sum_{t_{|A|}=0}^{\xi(x_{|A|})\wedge \eta(x_{|A|})}\\
& \qquad \qquad \qquad \qquad \qquad \qquad \qquad 
\qquad \qquad \qquad \qquad \times 
\prod_{i=1}^{|A|}
{g(t_i)}\cdot 
\frac{1}{|\Omega(A, \frac{S_A(\xi)+S_A(\eta)}{2}-S_A(t))|} 
\Bigr]. 
\end{align*}
This is symmetric with respect to $\xi$ and $\eta$. 

If we define $g$ as $g(k)=\delta_0(k)$, $k\in \mathbb{Z}_+$, 
then $G\equiv 1$ and the above argument yields that 
$P(Y_{n+1} = \eta | Y_n=\xi)$ 
is symmetric with respect to $\xi$ and $\eta$. 
Therefore, the uniform reshuffling model 
is doubly stochastic and 
its unique stationary distribution is give by 
the uniform distribution on $\Omega_{N}(L)$. 
Actually, this matches when $G\equiv 1$ is set in 
(\ref{stat}). 
\qed

\section{Proof of Theorem \ref{thm1}}
For the proof of Theorem \ref{thm1}, we adapt the proof of the 
equivalence of ensembles for the i.i.d. product measure 
(cf. \cite[Appendix 2]{KL}). 
In the following we assume that $g(0)>0$ and 
there exist $\alpha\in \mathbb{R}$ and 
$c_\alpha \in (0, \infty)$ such that 
$\lim\limits_{k\to \infty}\frac{g(k)}{k^\alpha} =c_\alpha$. 
We prepare several notations. 
Define $G(k)=\sum\limits_{j=0}^k g(j)$, 
$k\in \mathbb{Z}_+$ and 
$Q_n(s)=\sum\limits_{k\geq 0} k^n s^k G(k)$, 
$n \in \mathbb{Z}_+$. 
By the assumption on $g$, we have the following asymptotics 
of $G(k)$ as $k\to \infty$. 
\begin{equation}\label{GG}
G(k) \sim 
\begin{cases}
\frac{c_\alpha}{\alpha +1}k^{\alpha +1} & 
\text{ if } \alpha > -1, \\
{c_\alpha}{\log k} & 
\text{ if } \alpha = -1, \\
C_0 & 
\text{ if } \alpha < -1, \\
\end{cases}
\end{equation}
where $C_0=C_0(g)>0$ is a constant which depends on $g$. 
In particular, the radius of convergence of 
$Q_n(s)$ is $1$ and it holds that 
$\lim\limits_{s\uparrow 1}Q_n(s) = \infty$ 
for every $n\in \mathbb{Z}_+$ and $\alpha \in \mathbb{R}$. 
We define the exponential family of distributions 
$\{\nu_s(\, \cdot\, ) ; s \in [0, 1)\}$ on 
$\mathbb{Z}_+$ by
$\nu_s (k) = \frac{s^k G(k)}{Q_0(s)}$, 
$k \in \mathbb{Z}_+$. 
It is easy to see that 
$E^{\nu_s}[\eta(0)] = \frac{Q_1(s)}{Q_0(s)}$ is continuous, 
increasing in $s$ 
and diverges to infinity as $s\uparrow 1$. 
Hence, for every $K>0$ there exists unique 
${s}^*={s}^*(K) \in (0, 1)$ such that 
$E^{\nu_{s^*}}[\eta(0)] = K$. 
To examine 
the asymptotic behavior of $s^*(K)$ as $K \to \infty$, 
we use a Tauberian theorem of the following form 
(cf. {\cite[Corollary 1.7.3]{BGT}}). 
\begin{thm}
Let $\{a_k\}_{k\geq 0}$ be a sequence of non-negative 
numbers and assume that 
$A(s)=\sum\limits_{k=0}^\infty a_k s^k$ converges 
for $s\in [0, 1)$ and $\{a_k\}_{k\geq 0}$ is monotone. 
Then, the following are equivalent. 
\begin{itemize}
\item
$A(s) \sim \Gamma(\beta +1) (1-s)^{-\beta} h(\frac{1}{1-s})$ 
as $s \uparrow 1$ for $\beta >0$ and slowly varying function $h$. 
\item
$a_k \sim \beta k^{\beta-1} h(k)$ as $k\to \infty$ 
for $\beta >0$ and slowly varying function $h$. 
\end{itemize}
\end{thm}
\noindent
By this theorem and (\ref{GG}), the following asymptotics holds 
in the limit $s\uparrow 1$.  
\begin{equation}\label{ZZ}
Q_n(s) \sim 
\begin{cases}
\frac{c_\alpha \Gamma(n+\alpha + 2)}{\alpha +1}
(1-s)^{-(\alpha +n+2)} & 
\text{ if } \alpha > -1, \\
{c_\alpha \Gamma( n+1)}
(1-s)^{-(n+1)} \log \frac{1}{1-s}& 
\text{ if } \alpha = -1, \\
{C_0 \Gamma( n+1)}
(1-s)^{-(n+1)} & 
\text{ if } \alpha < -1, 
\end{cases}
\end{equation}
where we used the relation 
$\Gamma (\beta +1) =\beta \Gamma(\beta)$ for every $\beta>0$. 
Therefore, 
\begin{equation*}
E^{\nu_s}[\eta(0)] = \frac{Q_1(s)}{Q_0(s)} \sim 
\begin{cases}
\frac{\alpha + 2}{1-s} & 
\text{ if } \alpha > -1, \\
\frac{1}{1-s} & 
\text{ if } \alpha \leq -1, \\
\end{cases}
\end{equation*}
and this yields that 
\begin{equation}\label{sK}
s^*(K) = 
\begin{cases}
1- \frac{\alpha + 2}{K}(1+o(1)) & 
\text{ if } \alpha > -1, \\
1- \frac{1}{K}(1+o(1)) & 
\text{ if } \alpha \leq -1, \\
\end{cases}
\end{equation}
as $K\to \infty$. 
By these asymptotics we also have 
\begin{align}\label{varK}
\mathrm{Var}_{\nu_{s^*(K)}}(\eta (0)) 
= \frac{Q_2(s^*(K))}{Q_0(s^*(K))} -K^2 
=
\begin{cases}
\frac{1}{\alpha + 2}{K^2}(1+o(1)) & 
\text{ if } \alpha > -1, \\
{K^2}(1+o(1)) & 
\text{ if } \alpha \leq -1, \\
\end{cases}
\end{align}
as $K\to \infty$. 

Next, for each $s \in [0, 1)$ and 
$B \subset \mathbb{Z}$, let $\overline{\nu}_s^B$ 
be the product measure on 
$\mathbb{Z}_+^B$ whose one site marginal distribution equals to $\nu_s$. 
$\overline{\nu}_s^{B}
\bigl(\, \cdot\, \bigm| \Omega(B, L) \bigr)$
denotes the conditioned probability of $\overline{\nu}_s^B$ 
on the event that the total number of coins on $B$ 
equals to $L$. 
Then, we have the following key identity for (\ref{stat}). 
\begin{equation*}
\mu_{N, L}(\, \cdot\,) = 
\overline{\nu}_s^{\Lambda_N}
\bigl(\, \cdot\, \bigm| \Omega_N(L) \bigr). 
\end{equation*}
Notice that the right-hand side does not depend on the 
choice of the parameter $s$. 
Let $f:\mathbb{R}^A \to \mathbb{R}$ be a 
bounded continuous local function 
where $A$ is a finite subset of $\mathbb{Z}_+$. 
For every $N\in \mathbb{N}$ so that $\Lambda_N \supset A$, 
we have 
\begin{align}
\begin{split}\label{dec}
E^{\mu_{N, L}}\bigl[ f\bigl( 
\frac{\cdot}{a_N} \bigr)\bigr] 
& = 
\sum\limits_{\eta\in \mathbb{Z}_+^A} 
f\bigl( \frac{\eta}{a_N} \bigr) 
\overline{\nu}_s^{\Lambda_N}
\Bigl( \xi_A = \eta \bigm| \sum\limits_{x\in \Lambda_N}
\xi(x)= L \Bigr) \\
& = 
\sum\limits_{\eta\in \mathbb{Z}_+^A} 
f\bigl( \frac{\eta}{a_N} \bigr) 
\overline{\nu}_s^{A}(\eta) 
+
\sum\limits_{\eta\in \mathbb{Z}_+^A} 
f\bigl( \frac{\eta}{a_N} \bigr) 
\Bigl\{
\frac{\overline{\nu}_s^{\Lambda_N \setminus A}
\Bigl( \sum\limits_{x\in \Lambda_N\setminus A}
\xi(x)= L -S_A(\eta)\Bigr)}
{\overline{\nu}_s^{\Lambda_N }
\Bigl( \sum\limits_{x\in \Lambda_N}
\xi(x)= L \Bigr)}
-1\Bigr\} \overline{\nu}_s^{A}(\eta) 
\\
& =: I_1 +I_2,
\end{split} 
\end{align}
where $\xi_A$ denotes the configurations 
$\xi$ restricted on the set $A$. 
Now, we set $L=L_N$ 
where $\{L_N \}_{N\geq 1}$ be a sequence of positive integers 
that satisfies 
$\lim\limits_{N\to\infty}\frac{L_N}{N a_N}=T$ 
for some constant $T>0$ and divergent sequence 
$\{a_N\}_{N\geq 1}$. 
We also take 
$s$ in the right-hand side of (\ref{dec}) 
as $s_N^*:=s^*(\frac{L_N}{N})$. 
For this choice of $s$ we show that 
$I_1 \to \int_{\mathbb{R}_+^A} f(r) 
\overline{\mu}_{\alpha, T}^A(dr)$ 
and $I_2 \to 0$ as $N\to \infty$. 

For the proof of the convergence of $I_1$, we assume that 
$f$ is a function of one variable for notational simplicity. 
The general case can be proven by the similar manner since 
$\overline{\nu}_s^A$ 
is a product measure with the same marginal distribution. 
Firstly, we consider the case $\alpha>-1$. 
Let $R>0$. By (\ref{GG}) and (\ref{ZZ}), 
\begin{align*}
I_1 
& = 
\sum\limits_{k \geq 0} 
f\bigl( \frac{k}{a_N} \bigr) \frac{1}{Q_0(s_N^*)} (s_N^*)^k G(k) \\
& = 
\frac{1}{\Gamma(\alpha +2)}(1-s_N^*)
\sum\limits_{k = 0}^{[R a_N]} 
f\bigl( \frac{k}{a_N} \bigr)  (s_N^*)^k 
\bigl((1-s_N^*)k \bigr)^{\alpha+1} \\
& \qquad \qquad \qquad \qquad \qquad \qquad 
+ 
\frac{1}{\Gamma(\alpha +2)}(1-s_N^*)
\sum\limits_{k = [R a_N]+1}^\infty 
f\bigl( \frac{k}{a_N} \bigr)  (s_N^*)^k 
\bigl((1-s_N^*)k \bigr)^{\alpha+1} 
+o(1) \\
& =: I_{3}+I_4+ o(1), 
\end{align*}
as $N\to \infty$. We note that 
since $Q_0(s_N^*) \to \infty$ in the limit $N\to \infty$, 
a finite sum in $I_1$ is negligible and we can replace 
$G$ with the right-hand side of (\ref{GG}) 
with an error of $o(1)$. 
Then, by (\ref{sK}) and the condition on $L_N$, 
\begin{align*}
I_3 
& = 
\frac{1}{\Gamma(\alpha +2)}\frac{\alpha+2}{T}
\frac{1}{a_N} 
\sum\limits_{k = 0}^{[Ra_N]} 
f\bigl( \frac{k}{a_N} \bigr)  
\Bigl\{\bigl(1-\frac{\alpha+2}{T}\frac{1}{a_N}\bigr)^{a_N}
\Bigr\}^{\frac{k}{a_N}} 
\bigl(\frac{\alpha+2}{T}\frac{k}{a_N}\bigr)^{\alpha+1} 
+o(1) \\
& 
\to 
\frac{1}{\Gamma(\alpha+2)} \bigl(\frac{\alpha+2}{T}
\bigr)^{\alpha +2} \int_0^R f(r) e^{-\frac{\alpha+2} 
{T}r}r^{\alpha +1}dr, 
\end{align*}
as $N\to \infty$ where the convergence 
follows from Riemann integral. 
By taking the limit $R\to \infty$, 
the right-hand side converges to 
$\int_0^\infty f(r) \mu_{\alpha, T}(dr)$. 
For $I_4$, we have 
\begin{align*}
|I_4| 
& \leq 
\frac{1}{\Gamma(\alpha +2)}\frac{\alpha+2}{T}
\frac{1}{a_N} 
\sum\limits_{k =[Ra_N]+1}^{\infty} 
\bigm|\!\! f\bigl( \frac{k}{a_N} \bigr)\!\!\bigm| 
\bigl(1-\frac{\alpha+2}{T}\frac{1}{a_N}\bigr)^{k} 
\bigl(\frac{\alpha+2}{T}\frac{k}{a_N}\bigr)^{\alpha+1} 
+o(1) \\
& \leq 
C \frac{1}{a_N} 
\sum\limits_{k =[Ra_N]+1}^{\infty} 
e^{-\frac{\alpha+2}{T}\frac{k}{a_N}}
\bigl(\frac{k}{a_N}\bigr)^{\alpha+1} 
+o(1) \\
& 
\leq C'
\sum\limits_{j =[R]}^{\infty} 
e^{-\frac{\alpha+2}{T}j}
j^{\alpha+1} 
+o(1), 
\end{align*}
for every $N$ large enough where $C, C'$ are positive 
constants independent of $N$. 
By taking the limits $N\to \infty$ and $R\to\infty$, 
we obtain $I_4\to 0$. 

Secondly, we consider the case $\alpha=-1$. 
By (\ref{GG}) and (\ref{ZZ}) again, 
\begin{align*}
I_1 
& = 
\frac{1-s_N^*}{-\log (1-s_N^*)} 
\sum\limits_{k = 0}^{[Ra_N]} 
f\bigl( \frac{k}{a_N} \bigr)  (s_N^*)^k \log k
+ 
\frac{1-s_N^*}{-\log (1-s_N^*)} 
\sum\limits_{k = [Ra_N]+1}^\infty 
f\bigl( \frac{k}{a_N} \bigr)  (s_N^*)^k \log k
+o(1) \\
& =: I'_{3}+I'_4+ o(1), 
\end{align*}
as $N\to \infty$. 
${-\log (1-s_N^*)}=(1+o(1))\log a_N$ and this yields that 
\begin{align*}
I'_3 
& = 
\frac{1}{T a_N \log a_N}
\sum\limits_{k = 0}^{[Ra_N]} 
f\bigl( \frac{k}{a_N} \bigr)   
\Bigl\{\bigl(1-\frac{1}{T a_N}\bigr)^{a_N}
\Bigr\}^{\frac{k}{a_N}} 
{\log k}
+o(1) \\
& = 
\frac{1}{T a_N}
\sum\limits_{k = 0}^{[Ra_N]} 
f\bigl( \frac{k}{a_N} \bigr)   
\Bigl\{\bigl(1-\frac{1}{T a_N}\bigr)^{a_N}
\Bigr\}^{\frac{k}{a_N}} 
+ 
\frac{1}{T \log a_N}
\frac{1}{a_N}
\sum\limits_{k = 0}^{[Ra_N]} 
f\bigl( \frac{k}{a_N} \bigr)   
\Bigl\{\bigl(1-\frac{1}{T a_N}\bigr)^{a_N}
\Bigr\}^{\frac{k}{a_N}} 
\log \bigl(\frac{k}{a_N}\bigr)
+o(1). 
\end{align*}
In the limits $N\to \infty$ and $R\to \infty $, 
the first term of the right-hand side converges to 
$\frac{1}{T}\int_0^\infty f(r)
e^{-\frac{r}{T}}dr$ 
and the second term vanishes 
due to the extra factor $\frac{1}{\log a_N}$. 
$I'_4$ goes to $0$ in the same way as $I_4$. 
The case $\alpha<-1$ also follows from 
the similar argument. 

Next, for the convergence of $I_2$, 
we use the following local limit theorem. 
\begin{thm}\label{lclt}
Let $\{b_N\}_{N \geq 1}$ be a sequence of positive numbers 
which satisfies $\lim\limits_{N\to \infty}b_N=\infty$ and 
set $s_N^*:= s^*(b_N)$. 
For each $N\in \mathbb{N}$, 
$\{X_j^{(N)}\}_{j\in \Lambda_N}$ denotes a family of 
independent and identically distributed 
$\mathbb{Z}_+$-valued random variables 
with common distribution $\nu_{s_N^*}$. 
Then, for every finite set $B\subset \mathbb{Z}_+$, it holds that 
\begin{align}\label{lclt1}
\lim\limits_{N\to \infty} 
\sup\limits_{L \geq 0} \Bigm|\!
\sqrt{\sigma_N^2(N-|B|)} 
P \Bigl(
\sum_{j \in \Lambda_N \setminus B} X_j^{(N)} =L \Bigr) -
\frac{1}{\sqrt{2\pi}}
\exp\Bigl\{-\frac{( L-(N-|B|)b_N)^2}
{ 2 \sigma_N^2(N-|B|)}
\Bigr\}
\!\Bigm| =0, 
\end{align}
where $\sigma_N^2 = \mathrm{Var}(X_1^{(N)})$. 
\end{thm}
\noindent
The proof of this theorem is given later. 
By applying this theorem for 
$L=L_N$ and $b_N=\frac{L_N}{N}$, 
we have 
\begin{align*}
\frac{\overline{\nu}_{s_N^*}^{\Lambda_N \setminus A}
\Bigl( \sum\limits_{x\in \Lambda_N\setminus A}
\xi(x)= L_N -S_A(\eta)\Bigr)}
{\overline{\nu}_{s_N^*}^{\Lambda_N }
\Bigl( \sum\limits_{x\in \Lambda_N}
\xi(x)= L_N \Bigr)} 
& = 
\frac{\sqrt{N}}{\sqrt{N-|A|}}
\frac
{\exp\Bigl\{- 
\frac{(L_N-S_A(\eta) -(N-|A|)\frac{L_N}{N})^2}
{{2\sigma^2_N}{(N-|A|)}}\Bigr\} + o(1)}
{1 + o(1)}, 
\end{align*}
as $N\to \infty$ 
for every $\eta \in \mathbb{Z}_+^A$. 
Note that $o(1)$ terms do not depend on $\eta$. 
Set $D_A(k) = \{0, 1, \cdots, k\}^A$, $k\in \mathbb{Z}_+$. 
Then, the above asymptotics and (\ref{varK}) yield that 
\begin{align*}
\lambda_{N, R} := 
\sup\limits_{\eta \in D_A([Ra_N])} 
\Bigm| \! 
\frac{\overline{\nu}_{s_N^*}^{\Lambda_N \setminus A}
\Bigl( \sum\limits_{x\in \Lambda_N\setminus A}
\xi(x)= L_N -S_A(\eta)\Bigr)}
{\overline{\nu}_{s_N^*}^{\Lambda_N }
\Bigl( \sum\limits_{x\in \Lambda_N}
\xi(x)= L_N \Bigr)} -1 
\! \Bigm| 
\to 0, 
\end{align*}
as $N\to \infty$ for every $R>0$ and 
\begin{align*}
\sup_{N\geq 1}
\sup\limits_{\eta \in \mathbb{Z}_+^A} 
\Bigm|\! 
\frac{\overline{\nu}_{s_N^*}^{\Lambda_N \setminus A}
\Bigl( \sum\limits_{x\in \Lambda_N\setminus A}
\xi(x)= L_N -S_A(\eta)\Bigr)}
{\overline{\nu}_{s_N^*}^{\Lambda_N }
\Bigl( \sum\limits_{x\in \Lambda_N}
\xi(x)= L_N \Bigr)}-1 
\!\Bigm|\, 
\leq C'', 
\end{align*}
for some constant $C''>0$. 
Therefore, 
\begin{align*}
|I_2| 
& \leq 
\sum\limits_{\eta \in D_A([Ra_N])} 
\bigm|\! f\bigl( \frac{\eta}{a_N} \bigr) \!\bigm|
\Bigm|\! 
\frac{\overline{\nu}_{s_N^*}^{\Lambda_N \setminus A}
\Bigl( \sum\limits_{x\in \Lambda_N\setminus A}
\xi(x)= L_N -S_A(\eta)\Bigr)}
{\overline{\nu}_{s_N^*}^{\Lambda_N }
\Bigl( \sum\limits_{x\in \Lambda_N}
\xi(x)= L_N \Bigr)}
-1 \! \Bigm| \overline{\nu}_{s_N^*}^{A}(\eta) \\
& \qquad \quad + 
\sum\limits_{\eta\in 
\mathbb{Z}_+^A \setminus D_A([Ra_N])} 
\bigm|\! f\bigl( \frac{\eta}{a_N} \bigr) \!\bigm| 
\Bigm|\! 
\frac{\overline{\nu}_{s_N^*}^{\Lambda_N \setminus A}
\Bigl( \sum\limits_{x\in \Lambda_N\setminus A}
\xi(x)= L_N -S_A(\eta)\Bigr)}
{\overline{\nu}_{s_N^*}^{\Lambda_N }
\Bigl( \sum\limits_{x\in \Lambda_N}
\xi(x)= L_N \Bigr)}
-1\! \Bigm| \overline{\nu}_{s_N^*}^{A}(\eta) \\
& \leq 
\| f\|_{\infty}\lambda_{N, R} + 
C''
\sum\limits_{\eta\in 
\mathbb{Z}_+^A \setminus D_A([Ra_N])} 
\bigm|\! f\bigl( \frac{\eta}{a_N} \bigr) \!\bigm| 
\overline{\nu}_{s_N^*}^{A}(\eta).
\end{align*}
The first term of the right-hand side goes to $0$ 
as $N\to \infty$ for every $R>0$ and 
the second term goes to $0$ 
as $N\to\infty$ and $R\to \infty$ by the similar 
computation as the estimate of $I_4$ above. 
Hence, we obtain $I_2 \to 0$ 
and this completes the proof. 

If we assume the condition: 
$g(0)>0$ and $\sum\limits_{j\geq 0} g(j) <\infty$ 
for $g:\mathbb{Z}_+\to [0, \infty)$ instead, 
then the proof for the case $\alpha<-1$ above 
can be applied as it is. 

\qed
\begin{remark}
If $( L-(N-|B|)b_N)^2 \gg { 2 \sigma_N^2(N-|B|)}$
in (\ref{lclt1}) then the exponential term 
converges to $0$, 
rendering the local limit theorem ineffective.
For this reason, 
the local limit theorem in the form of Theorem \ref{lclt} 
was insufficient to prove the equivalence of ensembles 
in general settings, and a more refined version 
such as the Edgeworth expansion of at least the second order 
was necessary (cf. the proofs of Corollary 1.4 and 
Corollary 1.7 in \cite[Appendix 2]{KL}). 
On the other hand, such an expansion is not necessary and 
Theorem \ref{lclt} is sufficient in our case. 
Because we divided the summation 
$\sum\limits_{\eta \in \mathbb{Z}_+^A}$ of $I_2$ 
into 
$\sum\limits_{\eta \in D_A([Ra_N])}$ 
and  
$\sum\limits_{\eta \in 
\mathbb{Z}_+^A \setminus D_A([Ra_N])}$, 
the estimate for the latter part was reduced 
to the estimate of 
$\overline{\nu}_{s_N^*}^{A}
(\mathbb{Z}_+^A \setminus D_A([Ra_N]))$ 
which can be managed because we know the explicit 
form of ${\nu}_{s_N^*}$. 
\end{remark}

\medskip

\noindent
{\em Proof of Theorem \ref{lclt}.}
First of all, we note that 
Theorem \ref{lclt} corresponds to 
the local limit theorem for a triangular array 
of random variables 
since $\nu_{s_N^*}$ depends on the number of 
random variables $N$. 
Combining this with 
the fact that $s_N^*:= s^*(b_N) \to 1$ as 
$N\to \infty$, we cannot directly 
apply Theorem 1.3 or Theorem 1.5 in 
\cite[Appendix 2]{KL} which studied the refined version of 
the local limit theorem for the i.i.d random variables 
with common distribution $\nu_s$, $s\in (0, 1)$. 
Also, the known criteria of the 
local limit theorem for a triangular array 
of integer-valued random variables 
(e.g. \cite{DM95}, \cite{M91}) 
do not hold in our setting since 
$\sigma_N^2 \to \infty$ as $N\to \infty$. 
Therefore, we give the proof of the theorem 
according to the classical argument \cite[Chapter VII]{Pet}. 
For notational simplicity we only consider 
the case $B=\emptyset$. The modification for the general finite 
set $B\subset \mathbb{Z}_+$ is straightforward. 

Set $\widetilde{X}_j^{(N)} := \frac{1}{\sqrt{N\sigma^2_N}}
(X_j^{(N)} - b_N)$, $j \in \Lambda_N$ and 
$t_N(L):= \frac{1}{\sqrt{N\sigma^2_N}}
(L-Nb_N)$. 
We define $\phi_N(\theta) = 
E\bigl[\exp\{i\theta X_1^{(N)}\}\bigr]$ and 
$\psi_N(\theta) = 
E\bigl[\exp \bigl\{i\theta (\sum\limits_{j \in \Lambda_N} 
\widetilde{X}_j^{(N)})\bigr\}\bigr]$, $\theta \in \mathbb{R}$ 
where $i= \sqrt{-1}$. 
By the inversion formula, we have 
\begin{align*}
P \bigl(\sum\limits_{j\in \Lambda_N} X_j^{(N)} =L\bigr) 
= \frac{1}{2\pi \sqrt{N\sigma_N^2}} 
\int_{-\pi\sqrt{N\sigma_N^2}}^{\pi\sqrt{N\sigma_N^2}} 
e^{-i\theta t_N(L)} \psi_N(\theta) d\theta. 
\end{align*}
Therefore, 
\begin{align*}
& 2\pi \Bigm|\! \sqrt{N\sigma_N^2} 
P \bigl(\sum\limits_{j\in \Lambda_N} X_j^{(N)} =L\bigr) 
- \frac{1}{\sqrt{2 \pi}} e^{-\frac{1}{2}t_N(L)^2} \!\Bigm| \\
& = 
\Bigm|\! 
\int_{-\pi\sqrt{N\sigma_N^2}}^{\pi\sqrt{N\sigma_N^2}} 
e^{-i\theta t_N(L)} \psi_N(\theta) d\theta 
- 
\int_{-\infty}^{\infty} 
e^{-i\theta t_N(L)} e^{-\frac{1}{2} \theta^2} d\theta 
\!\Bigm| \\ 
& \leq 
\int_{|\theta| \leq R} 
|\psi_N(\theta) -e^{-\frac{1}{2} \theta^2} | d\theta 
+ 
\int_{R \leq |\theta| \leq \gamma \sqrt{N\sigma_N^2}} 
|\psi_N(\theta) | d\theta 
+ 
\int_{\gamma \sqrt{N\sigma_N^2} \leq 
|\theta| \leq \pi \sqrt{N\sigma_N^2}} 
|\psi_N(\theta) | d\theta 
+ 
\int_{|\theta| \geq R} 
e^{-\frac{1}{2} \theta^2} d\theta \\
& =: I_1+ I_2+I_3+I_4, 
\end{align*}
for every $R>0$ and $0<\gamma <\pi$. 
We show that the right-hand side converges to $0$ as 
$N\to \infty$ and $R\to \infty$. 

For $I_1$, assume that the law of $\sum\limits_{j \in \Lambda_N} 
\widetilde{X}_j^{(N)}$ converges to 
the standard normal distribution. 
Then, we have 
$\lim\limits_{N\to \infty} \psi_N(\theta) 
= e^{-\frac{1}{2} \theta^2}$ for every 
$\theta \in \mathbb{R}$ and we obtain 
$I_1 \to 0$ as $N\to \infty$ by the bounded 
convergence theorem. 
For the convergence of the law of 
$\sum\limits_{j \in \Lambda_N} 
\widetilde{X}_j^{(N)}$, 
we have only to show that 
$\sum\limits_{j\in \Lambda_N} 
E\bigl[
(\widetilde{X}_j^{(N)})^2; |\widetilde{X}_j^{(N)} 
| \geq \varepsilon\bigr] \to 0$ 
as $N\to \infty$ for every $\vare>0$ 
by Lindberg's central limit theorem 
(cf. \cite[Theorem 3.4.10]{Dur}). 
By (\ref{varK}), 
$|\widetilde{X}_j^{(N)} | \geq \vare$ implies that 
$X_j^{(N)} \geq \frac{1}{2}\vare \sqrt{N\sigma_N^2}$ 
for every $N$ large enough. Therefore, 
\begin{align*}
\sum\limits_{j \in \Lambda_N} 
E\bigl[
(\widetilde{X}_j^{(N)})^2; 
|\widetilde{X}_j^{(N)}| \geq 
\varepsilon\bigr] 
& \leq 
\frac{1}{\sigma_N^2} 
E\Bigl[(X_1^{(N)}-b_N)^2; 
X_1^{(N)} \geq \frac{1}{2}\vare \sqrt{N\sigma_N^2} 
\Bigr] \\
& \leq 
\frac{2}{\sigma_N^2} 
E\Bigl[({X}_1^{(N)})^2; 
{X}_1^{(N)} \geq \frac{1}{2}\vare \sqrt{N\sigma_N^2} \Bigr] 
+ 
\frac{2b_N^2}{\sigma_N^2} 
P \Bigl( 
{X}_1^{(N)} \geq \frac{1}{2}\vare \sqrt{N\sigma_N^2} \Bigr)\\
& =: J_1+J_2, 
\end{align*}
where the second inequality follows from the fact that 
$(a+b)^2 \leq 2a^2 +2b^2$ for every $a, b\in \mathbb{R}$. 
We first consider the case $\alpha> -1$ 
for the estimate of $J_1$. 
\begin{align*}
J_1 & 
= 
\frac{2}{\sigma_N^2} 
\sum\limits_{k \geq \frac{1}{2}\vare \sqrt{N\sigma_N^2}}
k^2 \frac{1}{Q_0(s_N^*)}(s_N^*)^k G(k)\\
& \leq 
\frac{C}{b_N^2}  
\sum\limits_{k \geq \frac{\vare}{4\sqrt{\alpha+2}} \sqrt{N b_N^2}}
k^2 \Bigl(\frac{\alpha+2}{b_N}\Bigr)^{\alpha+2} 
\Bigl\{\bigl(1-\frac{\alpha+2}{b_N} \bigr)^{b_N}\Bigr\}
^{\frac{k}{b_N}}k^{\alpha+1} \\ 
& \leq 
{C'}\frac{1}{b_N} 
\sum\limits_{ \frac{k}{b_N} \geq \frac{\vare}{4\sqrt{\alpha+2}} 
\sqrt{N}}
\Bigl(\frac{k}{b_N}\Bigr)^{\alpha+3} 
\Bigl\{\bigl(1-\frac{\alpha+2}{b_N} \bigr)^{b_N}\Bigr\}
^{\frac{k}{b_N}}, 
\end{align*}
for some constants $C, C'>0$ 
and every $N$ large enough 
where we used (\ref{ZZ}), (\ref{sK}) and (\ref{varK}) 
for the first inequality. 
The right-hand side goes to $0$ as $N\to\infty$ 
because 
\begin{align*}
\frac{1}{b_N} 
\sum\limits_{ \frac{k}{b_N} \geq a}
\Bigl(\frac{k}{b_N}\Bigr)^{\alpha+3} 
\Bigl\{\bigl(1-\frac{\alpha+2}{b_N} \bigr)^{b_N}\Bigr\}
^{\frac{k}{b_N}} 
\to 
\int_a^\infty r^{\alpha+3}e^{-(\alpha +2) r}dr <\infty, 
\end{align*}
as $N\to \infty$ for every $a \geq 0$. 
By the similar computation we obtain 
$J_1\to 0$ when $\alpha \leq -1$. 
For the estimate of $J_2$, we have 
\begin{align*}
J_2 \leq 
C P \Bigl( 
X_1^{(N)} \geq \frac{1}{2}\vare \sqrt{N\sigma_N^2} \Bigr) 
& \leq 
C \frac{2}{\vare \sqrt{N \sigma^2_N}} 
E[X_1^{(N)}] 
\leq 
\frac{C'}{\sqrt{N}}
\to 0, 
\end{align*}
as $N\to \infty$ where we used 
Markov's inequality and (\ref{varK}). 

Next, we consider $I_2$. 
By Taylor's theorem, there exists $\gamma_0>0$ such that 
for every $\theta \in \mathbb{R}$ which satisfies 
$|\frac{\theta}{\sqrt{N \sigma^2_N}}| \leq \gamma_0$, we have 
\begin{align*}
|E
\bigl[e^{i\theta \widetilde{X}_1^{(N)}}\bigr]| 
\leq 1-\frac{1}{4}
\bigl(\frac{\theta}{\sqrt{N \sigma^2_N}}\bigr)^2 
E\bigl[(X_1^{(N)}-b_N)^2\bigr] 
=1-\frac{\theta^2}{4N} \leq 
e^{-\frac{\theta^2}{4N}}. 
\end{align*}
Therefore, 
$|\psi_N(\theta)| \leq e^{-\frac{\theta^2}{4}}$ for every 
$\theta \in \mathbb{R}$ which satisfies 
$|{\theta}| \leq \gamma_0 {\sqrt{N \sigma^2_N}}$. 
Taking $\gamma$ in the definition of $I_2$ as $\gamma_0$, we 
obtain 
\begin{align*}
I_2 \leq 
\int_{R \leq |\theta| \leq \gamma_0 \sqrt{N\sigma_N^2}} 
e^{-\frac{\theta^2}{4}} d\theta 
\leq 2 
\int_{R}^\infty 
e^{-\frac{\theta^2}{4}} d\theta 
\to 0, 
\end{align*}
as $R\to \infty$. Similarly, we have 
$I_4\to 0$ as $R\to \infty$. 

The final task is the estimate of $I_3$. 
We take $\gamma$ in the definition of $I_3$ as 
$\gamma_0$ above. 
Let $0<\delta <1$ be fixed and define 
$\phi_N^\delta(\theta) := 
E\bigl[ (\delta e^{i\theta})^{X_1^{(N)}}\bigr] = 
\sum\limits_{k\geq 0} \delta^k e^{i\theta k} 
{{\nu}_{s_N^*}(k)}$. By the proof of 
\cite[Lemma 5.4]{LSV96}, we know that 
\begin{align*}
|\phi_N^\delta(\theta)| \leq 
\frac{1}{|1-\delta e^{i\theta}|}\Bigl\{
{\nu}_{s_N^*}(0)+ 
\sum\limits_{k\geq 0} 
|{{\nu}_{s_N^*}(k+1)}-{{\nu}_{s_N^*}(k)}|\Bigr\}.  
\end{align*}
We have 
${\nu}_{s_N^*}(0) = \frac{G(0)}{Q_0(s_N^*)} \to 0$ 
as $N\to \infty$ by (\ref{ZZ}) and (\ref{sK}). 
Also, 
\begin{align*}
|{{\nu}_{s_N^*}(k+1)}-{{\nu}_{s_N^*}(k)}| 
& = 
{\nu}_{s_N^*}(k) 
\bigm|\! \frac{{\nu}_{s_N^*}(k+1)}{{\nu}_{s_N^*}(k)}-1\!\bigm|  \\ 
& = 
{\nu}_{s_N^*}(k) 
\bigm|\! s_N^* \bigl\{\frac{g(k+1)}{G(k)}+1\bigr\}-1\!\bigm|\, 
\leq 
{\nu}_{s_N^*}(k) \Bigl\{ (1-s_N^*) + \frac{g(k+1)}{G(k)}\Bigr\}.  
\end{align*}
By the assumption on $g$, (\ref{GG}) and (\ref{sK}), 
for every $\vare>0$ there exists $N_0\geq 1$ and $C_1>0$ 
such that 
$(1-s_N^*) + \frac{g(k+1)}{G(k)} 
\leq \frac{C_1}{b_N}$ 
for every $N\geq N_0$ and $k\geq \vare b_N$. 
Moreover, 
$(1-s_N^*) + \frac{g(k+1)}{G(k)} 
\leq {C_2}$ 
for every $k\in \mathbb{Z}_+$ and $N\in \mathbb{N}$ 
where $C_2>0$ is a constant independent of $N$ and $k$. 
Therefore, 
\begin{align*}
\sum\limits_{k\geq 0} 
|{{\nu}_{s_N^*}(k+1)}-{{\nu}_{s_N^*}(k)}| 
\leq 
\sum\limits_{k=0}^{\vare b_N} C_2 {\nu}_{s_N^*}(k) 
+ 
\sum\limits_{k\geq \vare b_N} \frac{C_1}{b_N}{\nu}_{s_N^*}(k). 
\end{align*}  
The first term 
of the right-hand side converges to 
$C_3\int_0^\vare r^{\alpha+1}e^{-(\alpha +2) r}dr$ 
if $\alpha>-1$ and 
$C_3\int_0^\vare e^{- r}dr$ 
if $\alpha \leq -1$ 
for some $C_3>0$ by the similar computation as before. 
The second term is less than $\frac{C_1}{b_N}$ 
and this goes to $0$ as $N\to \infty$. 
As a result, for every $\vare> 0$ there exists $C_4(\vare)>0$ such that 
$C_4(\vare) \to 0$ as $\vare \to 0$ and 
\begin{align*}
|\phi_N^\delta(\theta)| \leq 
\frac{1}{|1-\delta e^{i\theta}|}\bigl\{ C_4(\vare) + \frac{1}{\pi} \gamma_0
\bigr\}, 
\end{align*}
for every $N$ large enough and every $\theta \in \mathbb{R}$. 
By taking the limit $\delta \uparrow 1$ and using the estimate 
$|1-e^{i\theta}| \geq \frac{2}{\pi}|\theta|$ for every 
$|\theta|\leq \pi$, we obtain 
\begin{align*}
|\phi_N\bigl( \frac{\theta}{\sqrt{N\sigma_N^2}}\bigr)| 
\leq 
\frac{\sqrt{N \sigma_N^2}}{|\theta|} \frac{\pi}{2} 
\bigl\{ C_4(\vare) + \frac{1}{\pi} \gamma_0\bigr\} 
\leq 
\frac{1}{\gamma_0} 
\frac{\pi}{2} 
\bigl\{ C_4(\vare) + \frac{1}{\pi} \gamma_0\bigr\}
= \frac{\pi}{2\gamma_0}C_4(\vare)+\frac{1}{2}, 
\end{align*}
for every $\theta \in \mathbb{R}$ so that 
$\gamma_0 \sqrt{N \sigma_N^2} \leq |\theta| \leq 
\pi \sqrt{N \sigma_N^2}$. 
Hence, by taking $\vare >0$ small enough, 
there exists $r<1$ such that 
$|\phi_N\bigl( \frac{\theta}{\sqrt{N\sigma_N^2}}\bigr)| 
\leq r $ 
for every $N$ large enough and 
$\theta \in \mathbb{R}$ so that 
$\gamma_0 \sqrt{N \sigma_N^2} \leq |\theta| \leq 
\pi \sqrt{N \sigma_N^2}$. This yields that 
\begin{align*}
I_3 = 
\int_{\gamma_0 \sqrt{N\sigma_N^2} \leq 
|\theta| \leq \pi \sqrt{N\sigma_N^2}} 
|\psi_N(\theta) | d\theta 
& = 
\int_{\gamma_0 \sqrt{N\sigma_N^2} \leq 
|\theta| \leq \pi \sqrt{N\sigma_N^2}} 
|\phi_N(\frac{\theta}{\sqrt{N \sigma_N^2}})|^N d\theta \\
& \leq 2\pi \sqrt{N \sigma_N^2} r^N \to 0, 
\end{align*}
as $N\to \infty$ and we can complete the proof of Theorem 
\ref{lclt}. 

\qed

\smallskip

\noindent
{\em Proof of Corollary \ref{cor1}.}\ 
By Proposition \ref{prop1}, we have 
\begin{align*}
\lim\limits_{n\to \infty} 
E_{\eta} \Bigl[ \frac{1}{N}
\bigm|\!\!\bigl\{x\in \Lambda_N; 
\frac{1}{a_N} {X_n^{(N)} (x)} \in (b, c) \bigr\}
\!\!\bigm|\Bigr] 
& = 
\lim\limits_{n\to \infty} 
\frac{1}{N} \sum\limits_{x\in \Lambda_N}
E_{\eta} \Bigl[ 
I \bigl( \frac{1}{a_N} X_n^{(N)} (x) 
\in (b, c)\bigr) \Bigr] \\
& = 
\frac{1}{N} \sum\limits_{x\in \Lambda_N}
\mu_{N, L_N}
\bigl( \frac{1}{a_N} \xi(x) \in (b, c)\bigr)\\
& = 
\mu_{N, L_N}
\bigl( \frac{1}{a_N} \xi(1) \in (b, c)\bigr). 
\end{align*}
Therefore, it is sufficient to show that 
$\lim\limits_{N\to \infty} 
\mu_{N, L_N}
\bigl( \frac{1}{a_N} \xi(1) \in (b, c)\bigr) 
= 
\mu_{\alpha, T} ((b, c))$. 
This follows from 
Theorem \ref{thm1} and 
the basic facts about the weak convergence of probability measures. 
The same is true for $\{Y_n^{(N)}\}_{n\geq 0}$ 
and $\{Z_n^{(N)}\}_{n\geq 0}$. 

\qed 

\section*{Acknowledgement}
The author thanks Mai Aihara 
for valuable discussions and 
her help on numerical simulations. 
This work was partially supported 
by JSPS KAKENHI Grant Number 22K03359. 


\end{document}